\documentclass[a4paper,10pt,twoside]{article}

\def\ifundefined#1{\expandafter\ifx\csname#1\endcsname\relax}

\usepackage{graphicx}
\usepackage{graphics}
\usepackage{mathrsfs}
\usepackage{float}
\usepackage{rotating}

\usepackage[T1]{fontenc}
\usepackage[english]{babel}
\usepackage[utf8]{inputenc}
\usepackage{amsthm,amsmath,amssymb,upref,amscd,amsfonts}  
\usepackage{euscript}
\usepackage{epsfig}
\usepackage{ae}
\usepackage{aecompl}
\usepackage[metapost]{mfpic}
\usepackage{enumerate}
\usepackage{enumitem}
\usepackage{url}
\usepackage[affil-it]{authblk}

\theoremstyle{plain}
\newtheorem{theorem}{Theorem}[section]

\newtheorem{lemma}[theorem]{Lemma}
\newtheorem{corollary}[theorem]{Corollary}
\newtheorem{theorem*}{Theorem}
\newtheorem{remark}{Remark}[section]

\theoremstyle{definition}

\newlength{\normalparindent}
\numberwithin{equation}{section}

\mathchardef\sa="303A

\newcommand{\R}{\ensuremath{\mathbf{R}}}

\newcommand{\C}{\ensuremath{\mathbf{C}}}

\newcommand{\Np}[3][p]{\ensuremath{\mathcal{N}_{#1}(#2 \, ; \, #3)}} 
\newcommand{\Npp}[3][p]{\ensuremath{\mathcal{N}_{#1}\bigl(#2 \, ; \, #3\bigr)}}
\newcommand{\Nppp}[3][p]{\ensuremath{\mathcal{N}_{#1}\biggl(#2 \, ; \, #3\biggr)}}

\newcommand{\Npw}[1][p]{\ensuremath{\mathcal{N}_{#1}}}

\newcommand{\Qw}[2]{\ensuremath{Q_{#1,#2}}} 

\newcommand{\Ckc}[1][\infty]{C^{#1}_c(\R^N \setminus \{0\})}
\newcommand{\Ck}[1][\infty]{C^{#1}_c(\R^N)}

\newcommand{\Yone}{\ensuremath{Y^{1,p}(\R^N)}}
\newcommand{\Yonez}{\ensuremath{Y^{1,p}_{0}(\R^N)}}

\newcommand{\Yonew}{\ensuremath{Y^{1,p}_{M}(\R^N)}}

\newcommand{\Wloc}{\ensuremath{W^{1,p}_{\mathrm{loc}}(\R^N \setminus \{ 0 \})}}
\newcommand{\Lloc}[1][1]{\ensuremath{L^{#1}_{\mathrm{loc}}(\R^N \setminus \{ 0 \})}}

\newcommand{\Xp}[1][1]{\ensuremath{X^p(\R^N)}}
\newcommand{\Xpett}[1][1]{\ensuremath{X^1(\R^N)}}

\newcommand{\Ew}[3][1]{\ensuremath{E({#1},{#2},{#3})}} 
\newcommand{\Ewt}[2][1]{\ensuremath{E({#1},{#2})}} 

\newcommand{\Ewn}[2][\xi]{\ensuremath{%
\int_{0}^{\infty} \!\! \int_0^{\infty} \Ew{\rho}{#1} \, #2 \, \frac{d#1}{#1} \, \frac{d\rho}{\rho}%
}}

\newcommand{\Awo}{\ensuremath{E}}
\newcommand{\Awt}{\ensuremath{\Sigma^{+}}}
\newcommand{\Awtm}{\ensuremath{\Sigma^{-}}}
\newcommand{\Awtpm}{\ensuremath{\Sigma^{\pm}}}

\newcommand{\ld}[1][x]{\ensuremath{\Phi(#1)}}
\newcommand{\vp}{\ensuremath{\varphi}}
\newcommand{\lc}[1][x]{\ensuremath{\Lambda(#1)}}
\newcommand{\lcw}{\ensuremath{\Lambda}}

\newcommand{\lcwz}{\ensuremath{\Lambda_0}}

\newcommand{\lcwm}{\ensuremath{\Lambda_{*}}}

\newcommand{\pval}{\ensuremath{\mathrm{p.v. }}} 

\newcommand{\lcwe}{\ensuremath{\lambda}}

\newcommand{\K}{\ensuremath{\mathscr{K}}} 
\newcommand{\KK}{\ensuremath{K}}

\newcommand{\apa}[2][\, \cdot \, ]{(\ensuremath{p( #2 \, \text{;} \, #1))}} 
\newcommand{\pa}[2][\alpha]{\ensuremath{p( #2 \, \text{;} \, #1)}}

\newcommand{\X}{\ensuremath{\mathscr{X}}} 
\newcommand{\DK}{\ensuremath{\mathscr{D}_{\K}}} 
\newcommand{\DKS}{\ensuremath{\mathscr{D}_{\K,\sigma}}} 
\newcommand{\DKK}{\ensuremath{\mathscr{D}_{\KK}}} 
\newcommand{\KZ}{\ensuremath{k_0}} 

\newcommand{\I}{\ensuremath{\Omega}} 
\newcommand{\RI}{\ensuremath{\R^{\I}}}

\renewcommand{\S}{\ensuremath{\mathcal{S}}}
\newcommand{\Ri}{\ensuremath{\mathcal{R}}}
\newcommand{\Ie}{\ensuremath{\mathcal{I}}}

\newcommand{\Bs}{\ensuremath{\mathscr{B}}}

\newcommand{\Clcwzs}{\ensuremath{C_{\lcwz}(s)}}

\newcommand{\dSy}{\ensuremath{\sqrt{1 + |\nabla \varphi(y)|^2}}}
\newcommand{\dSyw}{\ensuremath{\omega}}

\newcommand{\psif}{\ensuremath{\psi}}
\newcommand{\psifd}{\ensuremath{\widetilde{\psi}}}
\newcommand{\Ipsi}{\ensuremath{\Ie^{\psif}}}
\newcommand{\Rkpsi}{\ensuremath{R_k^{\psif}}}
\newcommand{\Lx}{\ensuremath{L_x}}
\newcommand{\Lxy}{\ensuremath{L_{xy}}}
\newcommand{\Ly}{\ensuremath{L_y}}

\newcommand{\kernd}[2]{\ensuremath{K_d({#1} , \, {#2})}}

\newcommand{\lu}{\ensuremath{{v}}}

\newcommand{\cop}{\ensuremath{\eta_{+}}}
\newcommand{\com}{\ensuremath{\eta_{-}}}
\newcommand{\coc}{\ensuremath{\eta_0}}

\newcommand{\copm}{\ensuremath{\eta_{\pm}}}

\newcommand{\pk}[1]{\ensuremath{\frac{\partial}{\partial_{{#1}_k}}}}
\newcommand{\pks}[1]{\ensuremath{\partial_k}}
\newcommand{\hlf}[1]{\ensuremath{\frac{{#1}}{2}}}

\newcommand{\Jp}{\ensuremath{J_{+}}}
\newcommand{\Jm}{\ensuremath{J_{-}}}
\newcommand{\Jpm}{\ensuremath{J_{\pm}}}
\newcommand{\Jc}{\ensuremath{J_{0}}}

\newcommand{\Jcs}{\ensuremath{J_{0}^{\mbox{\scriptsize s}}}}
\newcommand{\Jcns}{\ensuremath{J_{0}^{\mbox{\scriptsize ns}}}}

\newcommand{\Jpr}{\ensuremath{\Xi_{+}}}
\newcommand{\Jmr}{\ensuremath{\Xi_{-}}}
\newcommand{\Jpmr}{\ensuremath{\Xi_{\pm}}}

\newcommand{\Nhi}{\ensuremath{M}}
\newcommand{\Nh}{\ensuremath{M}}
\newcommand{\Ca}{\ensuremath{c_1}}
\newcommand{\Cb}{\ensuremath{c_3}}
\newcommand{\Cc}{\ensuremath{c_2}}

\newcommand{\qxy}{\ensuremath{\frac{|x|}{|y|}}}

\newcommand{\axy}{\ensuremath{a(x,y)}}

\newcommand{\bgp}[1]{\ensuremath{\biggl( {#1} \biggr)}}

\newcommand{\ck}{\ensuremath{C_K}}

\title{Single Layer Potentials on Surfaces with Small Lipschitz constant}
\author{Vladimir Kozlov}
\author{Johan Thim\footnote{Corresponding author: {\tt johan.thim@liu.se}}}
\author{Bengt Ove Turesson}
\affil{\small Department of Mathematics, University of Link\"{o}ping, Link\"{o}ping, Sweden}

\date{\today}

\begin{document}

\maketitle
\begin{abstract}
\noindent
This paper considers to the equation
\begin{equation*}
\int_{S}
\frac{U(Q)}{|P-Q|^{N-1}} \, dS(Q) = F(P) \text{,} \quad P \in S \text{,}
\end{equation*}
where the surface~$S$ is the graph of a Lipschitz function~$\vp$
on~$\R^N$, which has a small Lipschitz constant. The integral in the left-hand side
is the single layer potential corresponding to the Laplacian in~$\R^{N+1}$.
Let~$\lc[r]$ be a Lipschitz constant of~$\vp$ on the ball
centered at the origin with radius~$2r$.
Our analysis is carried out in local~$L^p$-spaces and local Sobolev spaces,
where~$1 < p < \infty$, and results are presented in terms of~$\lcw$.
Estimates of solutions to the equation are provided, which can be used
to obtain knowledge about the behaviour of the solutions near a point on the surface. 
The estimates are given in terms of seminorms.
Solutions are also shown to be unique if they are subject to 
certain growth conditions. 
Local estimates are provided and some applications are supplied. 
\end{abstract}

\ifundefined{thesis} \else
\pagenumbering{arabic}
\pagestyle{fancyplain}
\fi


\section{Introduction}
Let~$S$ be the graph in~$\R^{N+1}$ of a Lipschitz function~$\vp \colon \R^N \rightarrow \R$, where~$N \geq 2$.
We define~$\lcw$ to be a function on~$(0,\infty)$ such that
\begin{equation}
\label{eq:def_lc}
|\vp(x) - \vp(y)| \leq \lc[r] |x-y| \quad \text{for} \quad  |x|,\;|y| \leq 2r 
\end{equation}
for every~$r > 0$. The function~$\lcw$ is assumed to be increasing and bounded: 
\begin{equation}
\label{eq:i:lip_cond}
\lc[r] \leq \lcw_0 \quad \text{for every } r > 0.
\end{equation}
We will assume that~$\lcwz$ is sufficiently small.
One can choose~$\lc[r]$ to be the optimal constant in~(\ref{eq:def_lc}) and then~$\lcw_0$ is the 
(global) Lipschitz constant of~$\vp$. 

We consider the {\it single layer potential} on the surface~$S$: 
\begin{equation}
\label{eq:simplelayerpot}
\int_{S} \frac{U(Q)}{|P-Q|^{N-1}} \, dS(Q) \text{,} \quad P \in S \text{,}
\end{equation}
where~$dS$ is the Euclidean surface measure.
This object is important since, for instance, it appears when one applies the direct 
approach to solve Laplace's equation corresponding 
boundary integral equation; see, e.g., Hsiao and Wendland~\cite{wendland}.
If the surface is the hyperplane~$x_{N+1} = 0$, then we obtain the classical Riesz potential
of order one. 
The main objective of this article is to find a solution~$u$ to the equation
\begin{equation}
\label{eq:maineq}
\S u (x) = f(x) \text{,} \quad x \in \R^N \text{,}
\end{equation}
where
\[
\S u(x) = 
\int_{\R^N} \frac{u(y) \sqrt{1 + |\nabla \vp(y)|^2}}{|\ld[x] - \ld[y]|^{N-1}} \, dy \text{,} \quad x \in \R^N \text{.}
\]
Here,~$\ld[x] = (x,\vp(x))$ for~$x \in \R^N$, and~$\S u$ is the parametrization of the single layer potential
in~(\ref{eq:simplelayerpot}).

More specifically, we will consider equation~(\ref{eq:maineq}) 
for~$u \in \Lloc[p]$, where~$1 \leq p < \infty$, and~$f \in \Wloc$, where~$1 < p < \infty$.
We will formulate our results in terms of the family of
seminorms defined by
\[
\Np{u}{r} = \biggl( \frac{1}{r^N} 
	\int_{r \leq |x| < 2r} |u(x)|^p \, dx \biggr)^{1/p}
\text{,} \quad r > 0  \text{.}
\]
To simplify upcoming notation, let 
$\Qw{m}{n}(t) = t^{m}$ if~$0 < t \leq 1$, and
$\Qw{m}{n}(t) = t^{n}$ if~$t > 1$,
for non-negative~$m$ and~$n$.
For~$1 \leq p < \infty$, the Banach space~$\Xp[1]$ 
consists of all functions~$u$ that belong to~$\Lloc[p]$ and
satisfy 
\begin{equation}
\label{eq:defXp}
\int_{0}^{\infty} \Qw{N}{1}(\rho) \, \Np{u}{\rho} \, \frac{d\rho}{\rho} < \infty
\text{.}
\end{equation}
We take this expression as the norm on~$\Xp[1]$. This space is the
natural domain, in terms of the seminorms~$\Npw$, 
for the operator~$\S$ in the case that~$S$ is the hyperplane~$x_{N+1} = 0$; 
this is discussed further by the authors in~\cite{thim1}. We also remark that, if~$1 \leq p < N$,
then~$L^p(\R^N) \subset \Xp[1]$. If~$p \geq N$, there exist functions in~$L^p(\R^N)$ which do
not belong to~$\Xp$.

For~$1 < p < \infty$ and~$0 \leq \Nhi \leq N$,
the normed space~$\Yonew$ consists of all functions~$f$ in~$\Wloc$ such that
\begin{equation}
\label{eq:Yonew}
\int_0^{\infty} \Qw{\Nhi}{1}(\rho) \, \Np{\nabla f}{\rho} \, \frac{d\rho}{\rho} < \infty
\end{equation}
and~$\lim_{r \rightarrow \infty} \int_{S^{N-1}} f(r\theta) \, dS(\theta) = 0$, where~$S^{N-1}$ is the unit sphere in~$\R^N$.
The left-hand side of~(\ref{eq:Yonew})
defines the norm on this Banach space. 
The condition in~(\ref{eq:Yonew}) implies that the limit in the definition exists. This limit 
ensures that, e.g., constants do not belong to~$\Yonew$. The reason that functions of this type are excluded is that
we cannot expect to find a solution to~$\S u = f$ in this case; indeed, if~$f$ is a nonzero constant, 
then the solution~$u$ in
Theorem~\ref{t:i:exist} below would have to be~$u = 0$, which obviously does not solve~$\S u = f$ in any reasonable
way.
Furthermore, if~$f \in \Wloc$ such that~$|\nabla f| \in L^p(\R^N)$ for some~$N/\Nhi < p < N$,
then~$f \in \Yonew$.

In Section~\ref{s:proof_exist}, we prove the following existence result.

\begin{theorem}
\label{t:i:exist}
There exist positive constants\/~$\lcwm$,\/~$\Ca$,\/~$\Cc$, and\/~$\Cb$, depending only on\/~$N$
and\/~$p$, such that if\/~$\lcwz \leq \lcwm$ and 
if\/~$f \in \Yonew$ with\/~$M = N - \Cc \lcwz$ and~\/$1 < p < \infty$,
then the equation~{\rm(}\ref{eq:maineq}{\rm)} has a solution\/~$u \in \Xp$.
For\/~$r > 0$, this solution satisfies
\begin{equation}
\label{eq:i:est_sol}
\begin{aligned}
\Np{u}{r} \leq {} & \Cb \int_0^{r} \left( \frac{\rho}{r} \right)^{\Nhi} \Np{\nabla f}{\rho} \, \frac{d\rho}{\rho}\\
& + \Cb \int_r^{\infty} \exp \biggl( 
\Ca \int_r^{\rho} \lcw(\nu) \, \frac{d\nu}{\nu}
\biggr) \, \Np{\nabla f}{\rho} \, \frac{d\rho}{\rho}.
\end{aligned}
\end{equation}
\end{theorem}

\noindent
Even if the objects are different, Theorem~\ref{t:i:exist} is closely related to 
results obtained by V.A.~Kozlov and V.G.~Maz'ya for ordinary differential equations 
and ordinary differential equations with operator coefficients; see Section 6.4 in~\cite{kozlovmazyaSLE}
and Section 6.3 in~\cite{kozlovmazyaDEOC}. 
It is possible to use~(\ref{eq:i:est_sol}) to obtain two-weighted estimates for solutions to~(\ref{eq:maineq})
in weighted~$L^p$-spaces and weighted Sobolev spaces similar to those found in Section~7.5 in~\cite{kozlovmazyaDEOC};
see Section~8 in the authors' article~\cite{thim1} for an example of this procedure when the surface~$S$ is 
the hyperplane~$x_{N+1} = 0$. 
Furthermore, one can also compare
with the boundedness results for Riesz potentials in local Morrey-type spaces found in Burenkov et al.~\cite{burenkov2,burenkov1}
and references cited therein.

If~$\lcwz = 0$, we recover the same estimate in~(\ref{eq:i:est_sol}) for the solution
as in the case when~$S$ is the hyperplane~$x_{N+1} = 0$; compare with~(\ref{eq:NpR}) below. 
Furthermore, if~$\lcwz \rightarrow 0$, the condition in Theorem~\ref{t:i:exist} that~$f \in \Yonew$ 
reduces to 
the corresponding requirement for the hyperplane-case; 
see Corollary~\ref{c:cont_req_sol}. 

In Section~\ref{s:proof_uniq}, we prove that solutions to~(\ref{eq:maineq}) are unique if they
satisfy certain properties. More specifically, we have the following result.

\begin{theorem}
\label{t:i:uniq}
Suppose that\/~$u \in \Lloc[p]$,\/~$1 < p < \infty$, satisfies~{\rm(}\ref{eq:defXp}{\rm)} and 
\begin{equation}
\label{eq:est_uniq_infty}
\Np{u}{r} = 
O\biggl( \exp \biggl( -\Ca \int_1^r \lcw(\nu) \, \frac{d\nu}{\nu} \biggr) \biggr)   \quad \text{as } r \rightarrow \infty \text{,}
\end{equation}
and
\begin{equation}
\label{eq:est_uniq_zero}
\Np{u}{r} = 
O\bigl( r^{-\Nhi} \bigr)  \quad \text{as } r \rightarrow 0 \text{,}
\end{equation}
where\/~$\Ca$,\/~$\Cc$,\/~$\Nhi$, and\/~$\lcwm$ are in Theorem~\ref{t:i:exist}, and\/~$\lcwz \leq \lcwm$.
If\/~$\S u = 0$, then it follows that\/~$u = 0$.
\end{theorem}

\noindent It should be noted that the solution in Theorem~\ref{t:i:exist} satisfies the conditions in Theorem~\ref{t:i:uniq}; 
see Remark~\ref{r:soluniq} in Section~\ref{s:uniq}. 

In Section~\ref{s:assymp}, we prove a local version of Theorem~\ref{t:i:exist}.

\begin{theorem}
\label{t:i:local}
Let the constants\/~$\Ca$,\/~$\Cc$,\/~$\Nhi$, and\/~$\lcwm$ be as in Theorem~\ref{t:i:exist} and suppose that\/~$\lcwz \leq \lcwm$. 
Suppose also that\/~$u \in \Xp$, where\/~$1 < p < \infty$, satisfies~{\rm(}\ref{eq:est_uniq_zero}{\rm)} 
and\/~$\S u(x) = f(x)$ for\/~$|x| \leq 2r_0$, where\/~$r_0$ is a positive constant and\/~$f \in \Wloc$ satisfies
\begin{equation}
\label{eq:req_local}
\int_0^{2r_0} \Qw{\Nhi}{1}(\rho) \, \Np{\nabla f}{\rho} \, \frac{d\rho}{\rho} < \infty.
\end{equation}
Then\/~$\Np{u}{r}$ is bounded by
\begin{equation}
\label{eq:i:assymp}
\begin{aligned}
 & C \int_0^r \left({\frac{\rho}{r}}\right)^{\Nhi} \, \Np{\nabla f}{\rho} \, \frac{d\rho}{\rho}
  + C \int_r^{2r_0} \exp \biggl( \Ca \int_{r}^{\rho} \lcw(\nu) \, \frac{d\nu}{\nu} \biggr) \, \Np{\nabla f}{\rho}  \, \frac{d\rho}{\rho}\\
& \qquad + 
C \biggl( \| u \|_{\Xp} + \int_{r_0/2}^{2r_0} \Np{f}{\rho} \, \frac{d\rho}{\rho} \biggr)
\exp \biggl( \Ca \int_r^{r_0} \lcw(\nu) \, \frac{d\nu}{\nu}  \biggr),  
\end{aligned}
\end{equation}
for\/~$0 < r < r_0$, where\/~$C$ only depends on\/~$N$ and\/~$p$.
\end{theorem}

\noindent An application of this theorem can be found in Section~\ref{s:est_Np_ralpha}, where we show 
that if~$\Np{\nabla f}{r} \leq C r^{-\alpha}$ for small~$r$ and some~$\alpha \in (0,\Nhi)$ 
and~$\lcw(r) \rightarrow 0$ as~$r \rightarrow 0$, then~$\Np{u}{r} \leq C r^{-\alpha}$ (for small~$r$). 
We also show that if~$\lcw$ and~$\Np{\nabla f}{\, \cdot \, }$ satisfy Dini-type 
conditions at $0$, then~$\Np{u}{r} \leq C$ for small~$r$.

The existence and uniqueness of solutions to equation~(\ref{eq:maineq}) follows by applying
a fixed point theorem for locally convex spaces, which was developed by the authors in~\cite{thim2}.
To apply this theorem, we approximate the operator~$\S$ with a Riesz potential operator and control
the behaviour of the remainder.
The natural approach of approximating the surface~$S$ with the hyperplane~$x_{N+1} = 0$ does
not yield sufficiently strong estimates for our purposes. Instead, we use a weighted Riesz potential
to match the behaviour of~$\S u$ at the origin. The choice of weight is not obvious since we need
to estimate both the potential and its derivative.
The fact that the solution to the fixed point problem solves~(\ref{eq:maineq}) follows from results
derived earlier by the authors for the hyperplane case in~\cite{thim1}. A summary of
the hyperplane case can be found in Section~\ref{s:rieszpot}.


\section{Properties of the Single Layer Potential}
\subsection{Riesz Potentials on~\boldmath{$\R^N$}}
\label{s:prelim}
\label{s:rieszpot}
We start by recalling some properties of the potential~$\S u$ 
in the case when~$S$ is the hyperplane~$x_{N+1} = 0$.
These results were derived by the authors in~\cite{thim1}.
Equation~(\ref{eq:maineq}) reduces in this case to
\begin{equation}
\label{eq:maineq_riesz}
\Ie u(x) = \int_{\R^N} \frac{u(y)}{|x-y|^{N-1}} \,dy = f(x)\text{,} 
\quad x \in \R^N \text{,}
\end{equation}
where the operator $\Ie$ is the Riesz potential operator of order~1 defined 
for~$u$ in~$\Xpett$. The space~$\Xpett$ is the natural domain if~$\Ie u$ is 
interpreted as an absolutely convergent integral.
For solvability results in~$L^p$-spaces, we refer to Rubin~\cite{rubin} and references cited therein.
The following continuity properties for~$\Ie$ hold; see Theorem 1.3 in~\cite{thim1}.

\begin{theorem}
\label{t:cont_I1}
The operator\/~$\Ie$ maps\/~$\Xp[1]$ into\/~$\Wloc$ for\/~$1 < p < \infty$. Moreover,
there exist two constants\/~$C_1$ and\/~$C_2$, depending only on\/~$N$ and\/~$p$, such that
\begin{equation}
\label{eq:Np_est_I_1}
\Np{\Ie u}{r} \leq C_1r \int_{0}^{\infty} \Qw{N}{1}\left(\frac{\rho}{r}\right) 
	\, \Np{u}{\rho} \, \frac{d\rho}{\rho} \text{,} 
\end{equation}
and 
\begin{equation}
\label{eq:Np_est_diff_I_1}
\Np{\nabla \Ie u}{r} \leq C_2 \int_{0}^{\infty} \Qw{N}{0}\left(\frac{\rho}{r}\right) 
	\, \Np{u}{\rho} \, \frac{d\rho}{\rho} \text{,} 
\end{equation}
for every function\/~$u \in \Xp$ and every\/~$r > 0$.
In the first inequality,\/~$p=1$ is also allowed.
\end{theorem}
\noindent A solution to~(\ref{eq:maineq_riesz}) is given by
\begin{equation*}
\Ri f(x) = \frac{c_N}{N-1} \sum_{k=1}^{N} R_k \partial_k f(x) \text{,}
\quad x \in \R^N \text{,} 
\end{equation*}
where~$c_N = \Gamma\left( (N+1)/2 \right)\pi^{-({N+1})/{2}}$,~$\Gamma$ is the gamma function, and~$R_k$ is the~$k$th Riesz 
transform (cf.~Stein~\cite[p. 57]{stein}). 
For~$1 < p < \infty$, the space~$\Yone$ consists of all functions~$f$ in~$\Wloc$
such that 
\begin{equation}
\label{eq:y1pdef}
\| f \|_{\Yone} = \biggl| \int_{1 \leq |x| < 2} f(x) \, dx \biggr|
+ \int_{0}^{\infty} \Qw{N}{0}(\rho) \, \Np{\nabla f}{\rho} \, \frac{d\rho}{\rho} < \infty
\text{,}
\end{equation}
where this expression is the norm on~$\Yone$. The next theorem shows that 
the operator~$\Ri$ maps this space continuously into~$\Lloc[p]$; see Theorem~1.5 in~\cite{thim1}.

\begin{theorem}
\label{t:cont_R}
The operator\/~$\Ri$ is defined on\/~$\Yone$ for\/~$1 < p < \infty$, and there exists a constant\/~$C$, depending only on\/~$N$ and\/~$p$, 
such that
\begin{equation}
\label{eq:NpR}
\Np{\Ri f}{r} \leq
C \int_{0}^{\infty} \Qw{N}{0}\left( \frac{\rho}{r} \right) \, \Np{\nabla f}{\rho}
\, \frac{d\rho}{\rho}, \quad r > 0,
\end{equation}
for every function\/~$f \in \Yone$.
\end{theorem}
\noindent We define\/~$\Yonez$ as the proper subspace
of\/~$\Yone$ consisting of those functions\/~$f$ that satisfy
\begin{equation}
\label{eq:y1p0ineq}
\int_{0}^{1} \rho^N(1-\log \rho) \,  \Np{\nabla f}{\rho} \, \frac{d\rho}{\rho}
+ \int_{1}^{\infty} \Np{\nabla f}{\rho} \, d\rho < \infty 
\end{equation}
and $\lim_{r \rightarrow \infty} \int_{S^{N-1}} f(r\theta) \, d\theta = 0$.
The expression in~{\rm(}\ref{eq:y1p0ineq}{\rm)} is taken as the 
norm on\/~$\Yonez$. We have
the following solvability result; see Theorem 1.7 and 1.1 in~\cite{thim1}.

\begin{theorem}
\label{t:inv}
\label{t:injectivity}
\mbox{}
\begin{enumerate}
\item[{\rm(i)}]
The operator\/~$\Ri$ is bounded from\/~$\Yonez$ to\/~$\Xp[1]$ for\/~$1 < p < \infty$. 
\item[{\rm(ii)}]
If\/~$f \in \Yonez$, then there exists a solution\/~$u \in \Xp$ to~{\rm(}\ref{eq:maineq_riesz}{\rm)}
which satisfies
\[
\Np{u}{r} \leq
C \int_{0}^{\infty} \Qw{N}{0}\left( \frac{\rho}{r} \right) \, \Np{\nabla f}{\rho}
\, \frac{d\rho}{\rho}, \quad r > 0.
\]
\item[{\rm(iii)}]
Suppose that\/~$u$ is a locally integrable function such that
\begin{equation}
\label{eq:cond_uniq_plane}
\int_{|y| < 1} |u(y)| \, dy + \int_{|y| \geq 1} \frac{|u(y)| \, dy}{|y|^{N-1}} < \infty \text{.}
\end{equation}
If\/~$\Ie u = 0$, then it follows that\/~$u = 0$.
\end{enumerate}
\end{theorem}

\noindent It can be verified that the condition 
in~(\ref{eq:cond_uniq_plane}) coincides with the definition of~$\Xpett$; see~\cite{thim1}.


\subsection{Approximation of~\boldmath{$\S$}}
\label{s:appr_S}
For~$k = 1,2,\ldots,N+1$, 
let~$T_k$ be the singular integral operator defined by 
\[
T_{k} u(x) = \pval \int_{\R^N} 
\frac{(\ld[x] - \ld[y])_k}{|\ld[x] - \ld[y]|^{N+1}} \, u(y) \sqrt{1 + |\nabla \vp(y)|^2} \, dy
\text{,} \quad x \in \R^N \text{,}
\]
where~$u \in L^p(\R^N)$ with~$1 < p < \infty$. These operators are bounded on~$L^p(\R^N)$ 
for~$1 < p < \infty$; see for instance Dahlberg~\cite{dahlberg}. 
Moreover, analogously with Lemma~3.1 in the authors' article~\cite{thim1}, one can
show that if~$u \in \Lloc[p]$ satisfies
\begin{equation}
\label{eq:defTk}
\int_0^{\infty} \Qw{N}{0}(\rho) \, \Np{u}{\rho} \, \frac{d\rho}{\rho} < \infty,
\end{equation}
then~$T_k u$ is defined almost everywhere and
\begin{equation}
\label{eq:est_Tk}
\Np{T_k u}{r} \leq C \int_0^{\infty} \Qw{N}{0} \left( \frac{\rho}{r} \right) \, \Np{u}{\rho} \, \frac{d\rho}{\rho}, \quad r > 0,
\end{equation}
for~$k=1,2,\ldots,N+1$, where~$C$ only depends on~$N$ and~$p$.

Using Stokes' theorem, it is straightforward to show that~$\S u$ is weakly 
differentiable if~$u \in \Ck$, and that
\begin{equation}
\label{eq:diff_Su}
\partial_k \S u(x) = (1-N)\bigl( T_k u(x) + \partial_k \vp (x)T_{N+1} u(x) \bigr)
\text{,} \quad x \in \R^N \text{, } 
\end{equation}
for~$k =1,2,\ldots,N$.
Furthermore, one can show that~$\Ckc$ is dense in the Banach space defined by~{\rm(}\ref{eq:defTk}{\rm)},
so inequality~{\rm(}\ref{eq:est_Tk}{\rm)} and~(\ref{eq:diff_Su}) imply that~$\partial_k \S u$
is defined for~$u \in \Xp$ and given by~{\rm(}\ref{eq:diff_Su}{\rm)}.


\noindent To simplify the notation, let~$\dSyw(y) = \dSy$ for~$y \in \R^N$.
We wish to approximate~$\S u$ with a Riesz potential.  Put
\[
\psif(y) = 
\frac{|y|^{N+1}\dSyw(y)}{(|y|^2 + \vp(y)^2)^{(N+1)/2}},
\quad y \in \R^N \setminus \{ 0 \}.
\]
We define the operators~$\Ipsi$ and~$\Rkpsi$ by
\[
\Ipsi u = \Ie(\psif u), \quad u \in \Xp,
\]
and
\[
\Rkpsi u = R_k(\psif u), \quad u\text{ satisfying~(\ref{eq:defTk})}.
\]

For smooth~$u$, it follows from 
the fact that~$\pks{x} \Ipsi u = (1-N) c_N^{-1} \Rkpsi u$ (see Stein~\cite[p.~126]{stein}) 
and~(\ref{eq:diff_Su}) 
that
\begin{equation}
\label{eq:diff_S_I_1u}
\begin{aligned}
\partial_k (\Ipsi - \S )u &= (N-1) \bigl(
T_k u + \partial_k \vp T_{N+1} u - c_N^{-1} \Rkpsi u \bigr) \\
	&=  (N-1) \bigl( 
	( T_k u - c_N^{-1} \Rkpsi u )
	+ \partial_k \vp T_{N+1} u \bigr) \text{.}
\end{aligned}
\end{equation}
Equation~(\ref{eq:diff_S_I_1u}) also holds for~$u \in \Xp$. Indeed,~(\ref{eq:est_Tk}) and the corresponding
estimate for~$\Rkpsi$ imply that~(\ref{eq:diff_S_I_1u}) remains valid since~$C^{\infty}_c(\R^N)$ is dense 
in the involved spaces. 
\begin{lemma}
\label{l:S_I_bound_L2}
Let~$r > 0$ and define~$B_r = B(0 \, ; \, 2r)$.
There exists a constant~$C$, depending only on~$N$, such that
for every~$u \in L^2(B_r)$, that is supported in~$B_r$, 
\begin{enumerate}
\item[{\rm (i)}] $\| (T_{k} - c_N^{-1} \Rkpsi) u \|_{L^2(B_r)} 
	\leq C  \lc[r]^2 \, \| u \|_{L^2(B_r)}${\rm;}
\item[{\rm (ii)}] $\| T_{N+1} u \|_{L^2(B_r)} 
	\leq C  \lc[r] \, \| u \|_{L^2(B_r)}${\rm;}
\end{enumerate}
\end{lemma}

\begin{proof}
Put~$\psifd(y) = \psif(y)\,\dSyw(y)^{-1}$ for~$y \in \R^N \setminus \{ 0 \}$.
Define the operator~$T$ by
\[
Tu = T_k u - c_N^{-1} R_k ( \psifd u )
\]
for~$u \in L^2(B_r)$.
Let~$F_1$ be the function
\[
F_1(z) = \frac{(1+z^2)^{(N+1)/2}-1}{(1+z^2)^{(N+1)/2}},
\quad z \in \C \setminus \{ \pm i \}
\text{.}
\]
Then~$F_1$ is analytic in the band~$\{ \, z \in \C \sa |\mathrm{Im} \, z| < 1 \, \}$ and 
\[
\begin{aligned}
T_k u(x) - c_N^{-1} \Rkpsi u(x) = {} & 
	T u(x) + 
	  \int_{\R^N} \frac{x_k - y_k}{|x-y|^{N+1}} \left( 1 - \psifd(y) \right) u(y) \, \dSyw(y) \, dy\\
	= {} &
	-\int_{\R^N} 
	\frac{x_k - y_k}{|x-y|^{N+1}} F_1\biggl(\frac{\vp(x) - \vp(y)}{|x-y|} \biggr) u(y) \, \dSyw(y) \, dy\\
	& + \int_{\R^N} \frac{x_k - y_k}{|x-y|^{N+1}} \left( 1 - \psifd(y) \right) \dSyw(y) \, dy
\end{aligned}
\]
for~$1 \leq k \leq N$.
First, we consider~$T u$.
According to McShane's extension theorem (see~\cite{mcshane}), we may assume
that~$\vp$ is Lipschitz on~$\R^N$ with Lipschitz constant~$L = \lc[r]$
since we only need to consider~$T u (x)$ for~$x \in B_r$ and~$u$ has its support in~$B_r$ as well.
Let~$\gamma > 0$ and define
\begin{equation}
\label{eq:def_gamma}
A = \{ \, z \in \C \sa |\mathrm{Re} \, z| \leq L(1+\gamma) \text{, }  |\mathrm{Im} \, z| \leq L\gamma \, \} \text{.}
\end{equation}
We let~$\Gamma$ denote the rectangular boundary of the set~$A$ in~(\ref{eq:def_gamma}) and
assume that~$L = \lc[r]$ is sufficiently small, e.g.,~$L(1 + 2\gamma) \leq 1/2$, so that~$F_1$ is analytic in a
neighbourhood of the set in~(\ref{eq:def_gamma}). 

If~$\phi \sa \R \rightarrow \R$ is Lipschitz with Lipschitz-constant~$L$ 
and~$K$ is 
the Calder\'on-Zygmund kernel (see Stein~\cite{steinH}, Section~1.5)
\[
K(s,t) = \frac{1}{s-t} F_1\biggl( \frac{\phi(s) - \phi(t)}{s-t} \biggr)
\text{,} \quad s,t \in \R \text{,} \quad s \neq t \text{,}
\]
and~$V$ is the corresponding principal value operator:
\begin{equation}
\label{eq:pvo_tf}
V g(t) =
\pval \int_{-\infty}^{\infty} 
K(t,s) \, g(s) \, ds \text{,} \quad t \in \R \text{,}
\end{equation}
it follows
from well-known results for singular integral operators on Lipschitz curves 
that the operator~$V$ is bounded on~$L^2(\R)$: 
\begin{equation}
\label{eq:onedim_Tw}
\| V w \|_{L^2(\R)} \leq C (1 + 1/\gamma)^{3/2} \, \sup_{\omega \in \Gamma} |F(\omega)| \,  
\| w \|_{L^2(\R)} \text{,} 
\quad w \in L^2(\R) \text{,}
\end{equation}
where~$C$ is independent of~$\gamma$ and~$F$. The boundedness is a result by Calder\'on~\cite{calderon} for small~$L$ 
and Coifman, McIntosh, and Meyer~\cite{coifman} in the general case. 
The constant in~(\ref{eq:onedim_Tw}) can be derived from the argument
presented in Dahlberg~\cite[pp.~47--49]{dahlberg} together with the
optimal estimate given in David~\cite{david} for Cauchy integrals on Lipschitz
curves.

Employing the method of rotations and~(\ref{eq:onedim_Tw}) 
(see Dahlberg~\cite[pp.~49--50]{dahlberg} for the details), 
it follows that
\[
\| T u \|_{L^2(\R^N)} \leq C \, \sup_{z \in \Gamma} |F_1(z)|  
				\| u \|_{L^2(\R^N)}
\]
for~$u \in L^2(\R^N)$.
It can be verified directly that~$|F_1(z)| \leq C \, \lc[r]^2$ for~$z \in \Gamma$.
Moreover, since
\[
|1 - \psifd(y) | \dSyw(y) \leq C \, |\nabla \vp(y)|^2 \leq C \, \lc[r]^2, \quad y \in B_r,
\]
we obtain that 
\[
\biggl( \int_{\R^N} \! \left| \int_{\R^N} \frac{x_k - y_k}{|x-y|^{N+1}} 
	\left( 1 - \psifd(y) \right) u(y) \, \dSyw(y) \, dy \right|^2  dx \biggr)^{1/2} 
\leq
C \, \lc[r]^2 \| u \|_{L^2(\R^N)}.
\]
Thus, we have proved the inequality in~(i).
To prove~(ii), we utilise the same method but 
with the function~$F_2(z) = z(1+z^2)^{-(N+1)/2}$,~$z \in \C \setminus \{ \pm i \}$, instead of~$F_1$.
\end{proof}

\begin{corollary}
\label{c:S_I_bound_Lp}
Let\/~$1 < p < \infty$,~$r > 0$, and define~$B_r = B(0 \, ; \, 2r)$.
Then there exists a constant~$C$, which only depends on~$N$ and~$p$, such that
\begin{enumerate}
\item[{\rm (i)}] $\| (T_{k} - c_N^{-1} \Rkpsi) u \|_{L^p(B_r)} 
	\leq C  \lc[r]^2 \, \| u \|_{L^p(B_r)}${\rm;}
\item[{\rm (ii)}] $\| T_{N+1} u \|_{L^p(B_r)} 
	\leq C  \lc[r] \, \| u \|_{L^p(B_r)}${\rm;}
\end{enumerate}
for every function~$u \in L^p(B_r)$ that is supported in~$B_r$.
\end{corollary}

\begin{proof}
By Lemma~\ref{l:S_I_bound_L2}, we know that these estimates hold for~$p=2$. The operators
involved have Calder\'on--Zygmund kernels of the type
\[
K(x,y) = \frac{x_k-y_k}{|x-y|^{N+1}} F \left( \frac{\vp(x) - \vp(y)}{|x-y|} \right) \text{,}
\quad x \neq y \text{,}
\]
where~$F$ is analytic in some (complex) neighbourhood of the interval~$I = [-L,L]$
and~$L$ is the Lipschitz constant of~$\vp$.
Indeed, these kernels satisfy the properties in Section~1.5 of Stein~\cite{steinH}. In particular,
for~$c > 1$ and~$\delta > 0$, there exists a constant~$D$ such that
\begin{equation*}
\sup_{|z-y| < \delta} \; \int_{ |x-y| \geq c\delta} |K(x,y) - K(x,z)| \, dx \leq 
D. \quad
\end{equation*}
Here, 
$
D =  C \bigl( \|F\|_{L^{\infty}(I)} + \lc[r] \| F' \|_{L^{\infty}(I)} \bigr),
$
where~$C$ only depends on~$c$ and~$N$, and~$I = [-\lc[r],\lc[r]]$.
The function~$F$ is one of the two functions~$F_1$ and~$F_2$ in the proof of 
Lemma~\ref{l:S_I_bound_L2}. As in the proof of that lemma, we can assume
that~$\vp$ is Lipschitz continuous with constant~$\lc[r]$ since we only estimate the operator on~$B_r$
and~$u$ has its support in~$B_r$ as well.
Since
\[
|F_1'(s)| \leq (N+1)|s| \qquad \text{and} \qquad |F_2'(s)| \leq 1
\]
for~$s \in \R$,
the result now follows from Marcinkiewicz interpolation and the weak~$L^1$-estimate that
can be derived from the~$L^2$-estimate. Indeed, the $L^p$-norm for the interpolated
operator can be shown to have the form~$C \max \{ \, B, \, D_2 \, \}$, with~$B$ being the~$L^2$-norm 
of the operator and~$C$ depending only on~$N$ and~$p$.
\end{proof}


\subsection{Estimate of~\boldmath{$\Ipsi - \S$}}

To simplify the notation, we introduce the following quotients:
\[
\Lx = \frac{\vp(x)}{|x|}, 
\qquad
\Ly = \frac{\vp(y)}{|y|}, 
\quad \text{and} \quad
\Lxy = \frac{\vp(x) - \vp(y)}{|x-y|}, 
\]
defined for~$x \neq 0$,~$y \neq 0$, and~$x \neq y$, respectively.
It is clear that
\[
|\Lx| \leq \lcw(2^{-1}|x|), 
\quad
|\Ly| \leq \lcw(2^{-1}|y|), 
\quad
|\Lxy| \leq \max\{ \lcw(2^{-1}|x|) \, , \; \lcw(2^{-1}|y|) \}. 
\]
Furthermore, let us define~$\kernd{x}{y}$ by
\[
\kernd{x}{y} = \frac{\psif(y)}{|x-y|^{N-1}} - \frac{\dSy}{|\Phi(x) - \Phi(y)|^{N-1}}, \quad x,y \in \R^N, \quad x \neq y.
\]
This is the kernel in the operator~$\Ipsi - \S$.
Let us collect some properties of this kernel.

\begin{lemma}
\label{l:est_LxyLy}
If\/ $|x| < 2|y|$, then\/ $|\Lxy^2 - \Ly^2| \leq C\lcw(2^{-1}|x|)^2$,
and if\/ $|y| > 2|x|$, then\/ $|\Lxy^2-\Ly^2| \leq C \lcw(|y|)^2|x|/|y|$. 
\end{lemma}
\begin{proof}
First, observe that
\[
\begin{aligned}
\Lxy^2 - \Ly^2 &=
\Ly^2 \frac{2x\cdot y - |x|^2}{|x-y|^2} - 2 \frac{\vp(x)\vp(y)}{|x-y|^2} + \frac{\vp(x)^2}{|x-y|^2}.
\end{aligned}
\]
Hence, if~$|x| > 2|y|$, then
$
|\Lxy^2 - \Ly^2| \leq 8 \Ly^2 + 8|\Lx\Ly||y|/|x| + 4\Lx^2
\leq 16 \lcw(2^{-1}|x|)^2
$
and if~$|y| < 2|x|$, then
$
|\Lxy^2 - \Ly^2| 
\leq 6 \lcw(2^{-1}|y|)^2|x|/|y|.
$
\end{proof}

\begin{lemma}
\label{l:est_kernd}
Suppose that\/~$|x| > 2|y|$ or\/~$|y| < 2|x|$. Then there exists a constant\/~$C$, depending only on\/ $N$, such
that\/ $|\kernd{x}{y}| \leq C \psif(y)\Lxy^2|x-y|^{-(N-1)}$ 
and, if also\/ $\Lxy^2 - \Ly^2 \geq -1/2$,
\[
\left| \frac{\partial}{\partial x_k} \kernd{x}{y}  \right|
\leq
C \frac{\psif(y)}{|x-y|^N} \bigl( 
|\Lxy^2 - \Ly^2| + \lcw(2^{-1}|x|)|\Lxy|
\bigr).
\]
\end{lemma}

\begin{proof}
We rewrite~$\kernd{x}{y}$ and use the triangle inequality to obtain that
\[
\begin{aligned}
|\kernd{x}{y}| &= 
\frac{\psif(y)}{|x-y|^{N-1}} \left| 
1 - \frac{1}{\psi(y)} \frac{1}{(1 + \Lxy^2)^{(N-1)/2}}
\right|\\
&\leq 
\frac{\psif(y)}{|x-y|^{N-1}} \left( \left| 
1 - \frac{1}{(1 + \Lxy^2)^{(N-1)/2}}
\right| + \frac{\bigl| (1 + \Ly^2)^{(N+1)/2} - 1 \bigr|}{(1 + \Lxy^2)^{(N-1)/2}} \right)\\
&\leq C  \frac{\psif(y) \bigl( \Lxy^2 + \Ly^2 \bigr)}{|x-y|^{N-1}}.
\end{aligned}
\]
Moreover, since
\[
\begin{aligned}
\frac{\dSy}{|\ld[x] - \ld[y]|^{N+1}} &= 
\frac{\psif(y)}{|x-y|^{N+1}} \biggl(1 + \frac{\Lxy^2 - \Ly^2}{1+\Ly^2} \biggr)^{-(N+1)/2},
\end{aligned}
\]
we obtain that
\[
\pks{x} \kernd{x}{y}
=
\frac{(1-N) \psif(y)}{|x-y|^{N+1}} \biggl( 
\frac{x_k - y_k + \partial_k \vp (x)\bigl(\vp(x) - \vp(y)\bigr)}
	{\bigl(1 + \axy \bigr)^{(N+1)/2}}
- \bigl( x_k - y_k \bigr)
\biggr),
\]
where~$\axy = (\Lxy^2 - \Ly^2)/(1+\Ly^2)$. Thus,
\[
\begin{aligned}
|\pks{x} \kernd{x}{y} | &\leq
\frac{(N-1)\psif(y)}{|x-y|^{N}} \biggl( 
\biggl| \frac{1}{(1+\axy)^{(N+1)/2}} - 1 \biggr| + |\partial_k \vp(x)\Lxy|
\biggr) \\
&\leq
C \frac{\psif(y)}{|x-y|^{N}} \bigl( 
|\Lxy^2 - \Ly^2| + \lcw(2^{-1}|x|)|\Lxy|
\bigr), \\
\end{aligned}
\]
which is the estimate we want.
\end{proof}


\noindent We now have the tools necessary to derive the following important inequality.
\begin{theorem}
\label{t:est_diff_I_1_S}
Suppose that\/~$u \in \Ckc$. Then
\begin{equation}
\label{eq:est_diff_I_1_S}
\begin{aligned}
\Np{\pks{x} (\Ipsi - \S)u}{r} 
& \leq 
C \int_0^{\infty} \Ewt[r]{\rho} \, \Np{u}{\rho} \, \frac{d\rho}{\rho} 
\end{aligned}
\end{equation}
for~$k=1,2,\ldots,N$ and~$r > 0$, where~$\Ewt[r]{\rho}$ is defined by
\[
\Ewt[r]{\rho} = \begin{cases}
r^{-N} \rho^N \lc[r]^2, & 0 < \rho \leq r,\\
\lc[r] \bigl( \lc[r] r \rho^{-1} + \lc[\rho] \bigr), & \rho > r.\\
\end{cases}
\]
The constant~$C$ only depends on~$N$ and~$p$.
\end{theorem}

\begin{proof}
We introduce two cut-off functions~$\copm \in C^{\infty}(0,\infty)$ satisfying
\[
\cop(t) = \begin{cases}
1, & 0 \leq t \leq \frac{1}{8},\\
0, & t \geq \frac{1}{4},
\end{cases}
\qquad \text{and}
\qquad
\com(t) = \begin{cases}
0, & 0 \leq t \leq 4,\\
1, & t \geq 8.
\end{cases}
\]
Put~$\coc = 1 - \com - \cop$, so
\[
\coc(t) = \begin{cases}
1, & \frac{1}{4} \leq t \leq 4,\\
0, & 0 \leq t \leq \frac{1}{8} \text{ or } t \geq 8.
\end{cases}
\]
Using these cut-off functions, we split the integral in three parts:
\[
\begin{aligned}
(\Ipsi - \S)u(x) 
&= \int_{\R^N} \kernd{x}{y} \biggl( \cop \biggl(\qxy \biggr) + \com\bgp{\qxy} + \coc\bgp{\qxy}  \biggr)   u(y) \, dy\\
&= \Jp(x) + \Jm(x) + \Jc(x).
\end{aligned}
\]
In~$\Jpm(x)$, the kernel is smooth, so
\[
\begin{aligned}
\pks{x} \Jpm(x) &= \int_{\R^N} \biggl( \pk{x} \kernd{x}{y} \biggr) \copm\bgp{\qxy} u(y) \, dy
+ \Jpmr(x),
\end{aligned}
\]
where
\[
\Jpmr(x) = \int_{\R^N} \kernd{x}{y} \copm'\bgp{\qxy} \frac{x_k/|x|}{|y|} u(y) \, dy.
\]
It is clear from Lemmas~\ref{l:est_kernd} and~\ref{l:est_LxyLy} that
\[
\begin{aligned}
\Nppp{\pks{x} \Jpm}{r} &\leq 
\Nppp{\int_{\R^N} \copm\bgp{\qxy} \biggl| \pk{x} \kernd{x}{y} \biggr| |u(y)| \, dy}{r} 
 + \Np{\Jpmr}{r}\\
&\leq
C \Nppp{\int_{\R^N} \frac{g_{\pm}(x,y) |u(y)|}{\max\{{|y|^N, |x|^N\}}} \, dy}{r} 
 + \Np{\Jpmr}{r},
\end{aligned}
\]
where
\[
g_{-}(x,y) = \lcw^2\bgp{\hlf{|y|}} \qxy + \lcw\bgp{\hlf{|x|}} \lcw \bgp{\hlf{|y|}}, \quad \qxy \leq \frac{1}{4},
\]
and
\[
g_{+}(x,y) = \lcw^2\bgp{\hlf{|x|}}, \quad \qxy \geq 4,
\]
and~$g_{\pm}(x,y) = 0$ elsewhere. 
Suppose that~$r \leq |x| < 2r$
and~$0 \leq a \leq b \leq \infty$. 
Minkowski's inequality implies that
\[
\Nppp{\int_{a \leq \qxy \leq b} \frac{g_{\pm}(x,y) |u(y)|}{\max\{{|y|^N, |x|^N\}}} \, dy}{r} \leq
\int_{ar \leq |y| \leq 2br} \frac{\Np{g_{\pm}(\, \cdot \,,y)}{r} |u(y)|}{\max\{|y|^N,r^N\}} \, dy,
\]
Since~$\lcw$ is increasing,
\[
\Np{g_{-}(\, \cdot \,,y)}{r} \leq C \int_{|y|/2}^{|y|} \biggl( \lcw^2(\rho) \frac{r}{\rho} + \lcw(\rho)\lcw(r) \biggr) 
	\, \frac{d\rho}{\rho} , \quad y \in \R^N,
\]
and
\[
\Np{g_{+}(\, \cdot \,,y)}{r} \leq C \lcw(r)^2, \quad y \in \R^N.
\]
Now, if~$G$ is a measurable function, it is true that
\[
\begin{aligned}
\int_{a \leq |y| \leq b} \biggl( \int_{|y|/2}^{|y|} G(\rho) \, \frac{d\rho}{\rho} \biggr) |u(y)| \, dy
&= 
\int_{a}^{b} h(s) s^{N} \biggl(  \int_{s/2}^{s} G(\rho) \, \frac{d\rho}{\rho} \biggr) \, \frac{ds}{s}\\
&\leq
\int_{a/2}^{b} G(\rho) \int_{\rho}^{2\rho} s^{N} h(s) \, \frac{ds}{s} \, \frac{d\rho}{\rho}\\
&=
\int_{a/2}^{b} G(\rho) \, \Np[1]{u}{\rho} \, \frac{d\rho}{\rho},\\
\end{aligned}
\]
where~$h(s) = \int_{S^{N-1}} |u(s\theta)| \, d\theta$ for~$s > 0$, and the inequality follows from changing
the order of integration and over-estimating the domain of integration.

Thus, 
\begin{equation}
\label{eq:est_pkjm}
\begin{aligned}
\Nppp{\pks{x} \Jm}{r} \leq {} &
\int_{2r}^{\infty} \biggl( \lcw(\rho)^2 \frac{r}{\rho} + \lcw(\rho)\lcw(r) \biggr) \, \Np{u}{\rho} \, \frac{d\rho}{\rho}
 + \Np{\Jmr}{r}
\end{aligned}
\end{equation}
and
\begin{equation}
\label{eq:est_pkjp}
\begin{aligned}
\Nppp{\pks{x} \Jp}{r} \leq {} &
\int_0^{r/2} \biggl( \frac{\rho}{r}\biggr)^N \lcw^2(r) \, \Np{u}{\rho} \, \frac{d\rho}{\rho}
+ \Np{\Jpr}{r}.
\end{aligned}
\end{equation}
According to Lemma~\ref{l:est_kernd}, we can estimate the terms~$\Jpmr(x)$ by
\[
|\Jmr(x)| \leq C \int_{\frac{1}{8} \leq \qxy \leq \frac{1}{4}} \frac{|\kernd{x}{y}|}{|y|} |u(y)| dy
\leq C \int_{\frac{1}{8} \leq \qxy \leq \frac{1}{4}} \frac{|u(y)|\Lxy^2}{|y|^N} \, dy
\]
and
\[
|\Jpr(x)| \leq C \int_{4 \leq \qxy \leq 8} \frac{|\kernd{x}{y}|}{|y|} |u(y)| dy
\leq C \int_{4 \leq \qxy \leq 8} \frac{|u(y)|\Lxy^2}{|x||y|^{N-1}} \, dy,
\]
and analogously with the derivation of~(\ref{eq:est_pkjm}) and~(\ref{eq:est_pkjp}) above, we can then obtain a 
bound for~$\Np{\Jpmr}{r}$:
\begin{equation}
\label{eq:est_pkjrests_p}
\Np{\Jpr}{r} \leq
\lc[r]^2 
\int_{r/16}^{r/2} \Np{u}{\rho} \, \frac{d\rho}{\rho} 
\leq 
C \int_0^{\infty} \Ewt[r]{\rho} \, \Np{u}{\rho} \, \frac{d\rho}{\rho}
\end{equation}
and
\begin{equation}
\label{eq:est_pkjrests_m}
\Np{\Jmr}{r} \leq 
\lc[r]^2 
\int_{2r}^{16r} \Np{u}{\rho} \, \frac{d\rho}{\rho} 
\leq 
C \int_0^{\infty} \Ewt[r]{\rho} \, \Np{u}{\rho} \, \frac{d\rho}{\rho}.
\end{equation}

Turning our attention to~$\Jc$, we let~$B_r = \{ y \in \R^N \sa r/2 \leq |y| \leq 4r \}$ 
and~$\chi_{B_r}$ be the corresponding characteristic function. For~$r \leq |x| < 2r$,
\[
\begin{aligned}
\Jc(x) 
= {} &  
\int_{\R^N} \kernd{x}{y} \chi_{B_r}(y)u(y) \, dy \\
& +
\int_{\R^N} \coc\bgp{\qxy} \kernd{x}{y} (1 - \chi_{B_r}(y)) u(y) \, dy\\
= {} & \Jcs(x) + \Jcns(x).
\end{aligned}
\]
The integrand in~$\Jcns(x)$ is smooth since it is only nonzero when~$2|y| \leq |x| \leq 8|y|$
or~$2|x| \leq |y| \leq 8|x|$. Hence,
\[
\begin{aligned}
\pks{x} \Jcns = {} &
\int_{\R^N} \coc'\biggl( \qxy \biggr) \frac{x_k/|x|}{|y|} \kernd{x}{y} (1 - \chi_{B_r}(y))u(y) \, dy\\
&+ \int_{\R^N} \coc\biggl( \qxy \biggr) \frac{x_k/|x|}{|y|} \frac{\partial}{\partial x_k} \kernd{x}{y} (1 - \chi_{B_r}(y))u(y) \, dy.
\end{aligned}
\]
Using the same argument as above, we obtain that 
\begin{equation}
\label{eq:est_pkjcns}
\begin{aligned}
\Np{\pks{x} \Jcns}{r} \leq {} &
C \Nppp{\mathop{\int_{r/8 \leq |y| \leq r/2}}_{4r \leq |y| \leq 16r} \biggl( \frac{|\kernd{x}{y}|}{|y|}
+ | \pks{x} \kernd{x}{y}| \biggr) |u(y)| dy}{r}\\
\leq {} & C \int_0^{\infty} \Ewt[r]{\rho} \, \Np{u}{\rho} \, \frac{d\rho}{\rho}.
\end{aligned}
\end{equation}
Moreover,~(\ref{eq:diff_Su}) implies that
\[
\pks{x} \Jcs = (1-N)\bigl( 
	T_k( \chi_{B_r} u) - c_N^{-1} \Rkpsi (\chi_{B_r} u ) + \partial_k \vp T_{N+1} (\chi_{B_r} u)
\bigr),
\]
and an application of Corollary~\ref{c:S_I_bound_Lp} shows that
\[
\Np{\pks{x} \Jcs}{r} \leq C r^{-N/p} \lc[r]^2 \| u \chi_{B_r} \|_{L^p(\R^N)}.
\]
Thus, it is clear that
\begin{equation}
\label{eq:est_pkjcs}
\Np{\pks{x} \Jcs}{r} \leq C \int_0^{\infty} \Ewt[r]{\rho} \, \Np{u}{\rho} \, \frac{d\rho}{\rho}.
\end{equation}
The estimates 
in~(\ref{eq:est_pkjm}),~(\ref{eq:est_pkjp}),~(\ref{eq:est_pkjrests_p}),~(\ref{eq:est_pkjrests_m}),~(\ref{eq:est_pkjcns}), 
and~(\ref{eq:est_pkjcs}), imply that
the desired result holds.
\end{proof}

\begin{remark}
Analogously with the authors' proof of Lemma 3.1 in~{\rm\cite{thim1}}, it is possible to show that the to\/~$\partial_k (\Ipsi - \S)u(x)$ corresponding
``truncated'' operator\/~$T_{\epsilon} u(x)$ converges both pointwise and in\/~$\Lloc[p]$ to the right-hand 
side of~{\rm(}\ref{eq:diff_S_I_1u}{\rm)}
if\/~$u \in \Xp$ for\/~$1 < p < \infty$.
\end{remark}


\section{Reduction to a Fixed Point Problem}
\label{s:red_fixpt}

\subsection{The Fixed Point Equation}
Suppose that~$(\Ipsi - \S)u$ and~$f$ both belong to~$\Yonez$. Then, by Theorem~\ref{t:inv},
\[
\Ipsi \bigl( \psif^{-1} \Ri ( (\Ipsi - \S)u  + f) ) = (\Ipsi - \S)u + f \text{.}
\]
Formally, let~$\K$ be the operator given by
\begin{equation}
\label{eq:def_K}
\K(u) = \psif^{-1} \Ri ((\Ipsi - \S)u + f) \text{.}
\end{equation}
If~$u$ is a fixed point of~$\K$, 
then
\[
\Ipsi u = \Ipsi \K(u) = (\Ipsi - \S)u + f \text{.}
\]
Thus,~$u$ is a solution to~(\ref{eq:maineq}). To find a solution to~$\K(u) = u$, we will
employ the following fixed point theorem. We refer to a previous article
by the authors~\cite{thim1} for details and proofs. 


\subsection{A Fixed Point Theorem in Locally Convex Spaces}
\label{s:fixpt}
We let~$\X$ denote a locally convex topological space, where 
the topology
is given by a family~$\{ \pa{\,\cdot\,} \}_{\alpha \in \I}$ of seminorms that separates points.
We want to solve the equation
\begin{equation}
\label{eq:fixpt_eq}
\K(u) = u \text{,} \quad u \in \DK \text{,}
\end{equation}
where~$\K \colon \DK \rightarrow \DK$ is
a mapping defined on a subset~$\DK$ of~$\X$.
We assume that~$0 \in \DK$ and that
there exists an auxiliary linear operator~$\KK \colon \DKK \rightarrow \RI$,
where~$\DKK \subset \RI$ is a linear subspace. 
By~$\RI$ we denote the set of all real-valued functions on~$\I$, endowed with the
topology of pointwise convergence.

\subsubsection*{Existence of Fixed Points}
The operator~$\KK$ is subject to the following assumptions.
\begin{enumerate}[label=(\KK\arabic{*}), ref=(\KK\arabic{*})]
\item
\label{i:kkpos}
{\it Positivity.} The operator~$\KK$ is positive, i.e., if~$\eta \in \DKK$
is non-negative, then~$\KK \eta \geq 0$.
\item
\label{i:kkz}
{\it Fixed point inequality.}
The function~$\KZ(\,\cdot\,) = \pa[\,\cdot\,]{\K(0)} \in \DKK$, and
there exists a non-negative function~$z \in \DKK$ such that
\begin{equation}
\label{eq:tvsp_zineq}
z(\alpha) \geq \KK z(\alpha) + \KZ(\alpha) \text{,} \quad \alpha \in \Omega \text{.}
\end{equation}
\item
\label{i:kkclosed}
{\it Monotone closedness.}
The operator~$\KK$ is closed for non-negative, increasing sequences:
if~$\{ \eta_n \}$ is a non-negative sequence in~$\DKK$ such 
that~$\eta_n$ increases to~$\eta$, where~$\eta \leq z$,
and~$\KK \eta_n \rightarrow \zeta$, then~$\eta \in \DKK$ and~$\KK \eta = \zeta$.
\item
\label{i:kkdomain}
{\it Invariance.}
If~$\eta \in \DKK$ is non-negative and~$\eta \leq z$, 
then~$\KK \eta \in \DKK$. 
\end{enumerate}
The existence of a non-negative solution~$z$ to~(\ref{eq:tvsp_zineq}) enables us to prove
the existence of a non-negative solution to the equation
\begin{equation}
\label{eq:tvsp_sigmaeq}
\sigma(\alpha) = \KK \sigma(\alpha) + \KZ(\alpha) \text{,} \quad \alpha \in \Omega \text{,}
\end{equation}
which is minimal in the sense that if~$\eta \in \DKK$ is another non-negative solution
to~{\rm(}\ref{eq:tvsp_sigmaeq}{\rm)}, then~$\sigma \leq \eta$; see Lemma~1 in~\cite{thim2}.

Suppose that the operator~$\K$ maps~$\DK$ into~$\DK$.
We let~$\sigma$ be the minimal solution to~(\ref{eq:tvsp_sigmaeq}),
and put
\[
\DKS = \{ u \in \DK \sa \pa{u} \leq \sigma(\alpha) \text{ for every } \alpha \in \I \} \text{.}
\]
We require the following properties to hold.
\begin{enumerate}[label=(\ensuremath{\K}\arabic{*}), ref=(\ensuremath{\K}\arabic{*})]
\item
\label{i:DKS_est_K_KK}
{\it Subordination to~$\KK$.}
If~$u,v$ belong to~$\DKS$,
then~$\pa[\,\cdot\,]{u-v}$ belongs to~$\DKK$, and we have
\begin{equation}
\label{eq:est_K_KK}
\pa{\K(u) - \K(v)} \leq \KK\apa{u - v}(\alpha) \text{,} \quad \alpha \in \I \text{.}
\end{equation}
\item
\label{i:DKS_closed}
{\it Closedness of~$\DKS$.}
If~$\{v_k\}_{k=0}^{\infty}$ is a sequence in~$\DKS$ such 
that~$v_0 = 0$ and
\begin{equation}
\label{eq:sumcond}
\sum_{k=0}^{\infty} \pa{v_{k+1}-v_k} \leq \sigma(\alpha) \text{,} \quad \alpha \in \I \text{,}
\end{equation}
then the limit of~$v_k$ exists and belongs to~$\DKS$.
\end{enumerate}
Since~$0 \in \DKS$,~\ref{i:DKS_est_K_KK} implies that~$\pa[\,\cdot\,]{u} \in \DKK$
for~$u \in \DK$. 

\begin{theorem}
\label{t:tvsp_fixpt}
Suppose that\/~$\KK$ satisfies~\ref{i:kkpos} to~\ref{i:kkdomain} 
and that\/~$\K$ satisfies~\ref{i:DKS_est_K_KK} and~\ref{i:DKS_closed}.
Then there exists a fixed point of\/~$\K$ in\/~$\DKS$.
\end{theorem}

\subsubsection*{Uniqueness of Fixed Points}
Suppose that the operator~$\K$ maps~$\DK$ into itself.
We assume that the following conditions hold.
\begin{enumerate}[label=(\Roman{*}), ref=(\Roman{*})]
\item
\label{i:KKnu}
If~$u \in \DK$, then~$\KK^{n} \apa{u}$ is defined and belongs to~$\DKK$ for~$n=1,2,\ldots$,
and~$\lim_{n \rightarrow \infty} \KK^{n} \apa{u}  = 0$. 
\item
\label{i:KKfp}
If~$u \in \DK$ and~$\eta \in \DKK$ satisfy~$0 \leq \eta(\alpha) \leq \pa{u}$ 
for every~$\alpha \in \I$, 
then~$\KK \eta$ belongs to~$\DKK$. 
\item
\label{i:DK_est_K_KK}
If~$u,v$ belong to~$\DK$,
then the function~$\pa[\,\cdot\,]{u-v}$ belongs to~$\DKK$, and~(\ref{eq:est_K_KK}) holds.
\end{enumerate}

\begin{theorem}
\label{t:tvsp_fixptu}
Suppose that the operators\/~$\K$ and\/~$\KK$ satisfy~\ref{i:kkpos} 
and~\ref{i:KKnu} to~\ref{i:DK_est_K_KK}, respectively.
Then there exists at most one fixed point of\/~$\K$ in\/~$\DK$.
\end{theorem}


\subsection{Construction of~\boldmath{$\K$} and~\boldmath{$\KK$}} 
\label{s:fixpt_defs}
To apply the fixed point theorem, we define an operator~$\K$ formally by~(\ref{eq:def_K}) and
an auxiliary linear operator~$\KK$, and verify the properties
in~\ref{i:kkpos}--\ref{i:kkdomain} along with~\ref{i:DKS_est_K_KK} and~\ref{i:DKS_closed}. 
The locally convex space~$\X$ will be~$\Lloc[p]$ with seminorms~$\pa{\,\cdot\,}$ given by~$\Np{\,\cdot\,}{r}$,~$r > 0$.
Let
\[
\Ew[r]{\rho}{\xi} = \Qw{N}{0}\left(\frac{\rho}{r}\right)
\Ewt[\rho]{\xi} 
\text{,} \quad r,\rho,\xi > 0 \text{,}
\]
where~$\Ewt[\rho]{\xi}$ is defined as in Theorem~\ref{t:est_diff_I_1_S}.
We define the linear operator~$\KK$ by
\begin{equation}
\label{eq:def_kk_r}
\begin{aligned}
\KK \zeta(r) &= \ck \int_0^{\infty} \!\! \int_0^{\infty} 
\Ew[r]{\rho}{\xi}
\, \zeta(\xi) \, \frac{d\xi}{\xi} \, \frac{d\rho}{\rho} \text{,} \quad r > 0 \text{,}
\end{aligned}
\end{equation}
with domain~$\DKK$ given by those measurable~$\zeta$ that satisfy
\begin{equation}
\label{eq:cond_domK}
\Ewn{|\zeta(\xi)|}
< \infty \text{.}
\end{equation}
The constant~$\ck$ is the one given in~(\ref{eq:def_ck}) below. 
Since
\[
\Ew[r]{\rho}{\xi} \leq \max \bigl\{ \, 1, \, r^{-N} \, \bigr\} \Ew[1]{\rho}{\xi}
\text{,} \quad r,\rho,\xi > 0 \text{,}
\]
it follows that~$\KK \zeta$ is defined if~$\zeta \in \DKK$.
This of course implies that~$\KK \zeta(r)$ is finite for every~$r > 0$ if~$\zeta \in \DKK$.

By~$\DK$, we denote the set of those~$u \in \Lloc[p]$ such that
\begin{equation}
\label{eq:cond_domK_Np}
\Ewn{\Np{u}{\xi}}
< \infty \text{.}
\end{equation}
This condition is equivalent to requiring that~$\Np{u}{\, \cdot \,} \in \DKK$. 
Obviously~$\K(u)$ is defined for~$u \in \Ckc$, and the following two
lemmas show that~$\K$ can be extended continuously from~$\Ckc$ to~$\DK$.

\begin{lemma}
\label{l:Npu_DKK_Yone}
Let\/~$u \in \Ckc$. Then\/~$(\Ipsi - \S)u \in \Yone$ and 
\begin{equation}
\label{eq:est_RIu}
\Np{\Ri ((\Ipsi - \S)u)}{r} \leq \KK ( \Np{u}{\, \cdot \,} )(r)
	\text{,} \quad r > 0 \text{.}
\end{equation}
\end{lemma}

\begin{proof}
Since~$u$ has compact support,~$u \in \Xp$. Hence,~$(\Ipsi - S)u$ is defined
and~$\partial_k (\Ipsi - \S)u$,~$k=1,\ldots,N$, is
given by~(\ref{eq:diff_S_I_1u}). 
It is a straightforward calculation to verify that~$(\Ipsi - \S)u$ belongs to~$\Yone$ 
by using~(\ref{eq:est_diff_I_1_S}) and changing the order of integration. By 
Theorem~\ref{t:cont_R},~$\Ri$ is defined on~$\Yone$ and
\begin{equation}
\label{eq:def_ck}
\begin{aligned}
\Np{\Ri ((\Ipsi - \S)u)}{r}
	&\leq 
	C \, \int_0^{\infty} \Qw{N}{0} \left(\frac{\rho}{r}\right) 
		\Npp{\nabla \bigl((\Ipsi - \S)u\bigr)}{\rho} \, \frac{d\rho}{\rho}\\
	& \leq \ck \int_0^{\infty} \! \! \int_0^{\infty} 
		\Ew[r]{\rho}{\xi}
\, \Np{u}{e^{-\xi}} \, \frac{d\xi}{\xi} \, \frac{d\rho}{\rho}\\
&= \KK(\Np{u}{\, \cdot \,})(r) 
\end{aligned}
\end{equation}
for every~$r > 0$, where we used Theorem~\ref{t:est_diff_I_1_S} in the last inequality.
\end{proof}

\begin{corollary}
Let\/~$u,v \in \Ckc$ and\/~$f \in \Yone$. Then 
\begin{equation}
\label{eq:NpKuKv}
\Np{\K(u) - \K(v)}{r} \leq \KK(\Np{u-v}{\, \cdot \,})(r) \text{,} \quad r > 0 \text{.}
\end{equation}
\end{corollary}

\begin{lemma}
\label{l:ext_to_B}
Let\/~$f \in \Yone$. The operator\/~$\K$ in~{\rm(}\ref{eq:def_K}{\rm)} can be extended 
to\/~$\DK$ so that~{\rm(}\ref{eq:NpKuKv}{\rm)} holds for all\/~$u,v \in \DK$.
\end{lemma}

\begin{proof}
Let~$\Bs$ be the space of functions in~$\DK$ with topology defined by the norm
\begin{equation}
\label{eq:norm_in_B}
\| u \|_{\Bs} = \Ewn{\Np{u}{\xi}} 
\text{.}
\end{equation}
This is a Banach space and
one can check that~$\Ckc$ is a dense subspace of~$\Bs$. 
The operator~$\Ri ((\Ipsi - \S))$ is defined for~$u \in \Ckc$ and
\begin{equation}
\label{eq:A_bnd}
\Np{\Ri ((\Ipsi - \S)u)}{r} \leq C \max \{1,\, r^{-N} \} \, \| u \|_{\Bs} \text{.}
\end{equation}
By density, this allows us to extend~$\Ri ((\Ipsi - \S))$ uniquely to all of~$\Bs$ so 
that~(\ref{eq:A_bnd}) holds
for all~$u$ in~$\Bs$. Obviously this gives an extension of the operator~$\K$ to~$\DK$ as well.
\end{proof}


\subsection{Verification of~($\mathbf{K2}$)}
\label{s:aux_eq}
Next, we are going to show that~\ref{i:kkz} holds. 
More specifically, we will prove that there exists a function~$z \in \DKK$ such that
\begin{equation}
\label{eq:z}
\begin{aligned}
z(r) \geq {} & \ck \biggl( \int_0^{\infty} \!\! \int_0^{\infty} 
	\Ew[r]{\rho}{\xi}
	\, z(\xi) \, \frac{d\xi}{\xi}  \, \frac{d\rho}{\rho}
 +  \int_0^{\infty} 
	\Qw{N}{0}\left(\frac{\rho}{r}\right) \, \Np{\nabla f}{\rho} \, \frac{d\rho}{\rho}  \biggr)
\end{aligned}
\end{equation}
for every~$r > 0$.
This will imply that~\ref{i:kkz} is satisfied since~$\Np{\K(0)}{r}$ is bounded by
the second term in the right-hand side of~(\ref{eq:z}) for~$r > 0$.
\newcommand{\f}{\ensuremath{\zeta}}

We construct a solution in the following manner.
Let~$\Ca,\,\Cb,\,$ and~$\Cc$ be positive constants depending only on~$N$ and~$p$. 
We require that~$\Cc \lcwz \leq (N-1)/2$ and~$\Ca \lcwz \leq 1/2$.
Put~$\Nh = N - \Cc \lcwz$ and~$\lcwe(\nu) = \lcw(\exp(-\nu))$ for~$\nu \in \R$. 
The function~$\lu$ is given by
\begin{equation}
\label{eq:lu}
\begin{aligned}
\lu(t) = \Cb \int_{-\infty}^{t} \exp\biggl( \Ca \int_{s}^{t} \lcwe(\nu) \, d\nu \biggr) \f(s) \, ds
+ \Cb \int_t^{\infty} \exp \bigl( \Nh(t-s) \bigr) \f(s) \, ds,
\end{aligned}
\end{equation}
where~$\f(s) = \Np{\nabla f}{\exp(-s)}$ for~$s \in \R$.
We require that~$\f$ satisfies 
\begin{equation}
\label{eq:req_lu}
\int_{-\infty}^{0} \exp\biggl( \Ca \int_{s}^{0} \lcwe(\nu) \, d\nu \biggr) \f(s) \, ds
+ \int_0^{\infty} \exp \bigl( -\Nh s \bigr) \f(s) \, ds < \infty.
\end{equation}
It is clear that if~$\f$ satisfies~\ref{eq:req_lu}, then~$\lu(t)$ is finite for every~$t \in \R$. Moreover,
if~$f \in \Yonew$, then~(\ref{eq:req_lu}) is also valid.

\newcommand{\intlam}[2]{\ensuremath{\exp \biggl( \Ca \int_{{#1}}^{{#2}} \lcwe(\nu) \, d\nu \biggr)}}
\newcommand{\intl}{\ensuremath{D}}

We change variables in the definition of the operator~$\KK$ to~$r = e^{-t}$, $\rho = e^{-\tau}$, and~$s = e^{-\sigma}$, where~$t,\tau,\sigma \in \R$.
We allow earlier functions of the variables~$r,\rho,s$ to keep the same name when it is clear from the context what we are referring to.
The action of~$\KK$ on~$\lu$ can be expressed with help from the functions
\begin{equation}
\Awo(\tau, \sigma) = \Biggl\{
\begin{aligned}
&\lcwe^2(\tau) \exp(N(\tau - \sigma)), & \tau \leq \sigma, \\
&\lcwe(\tau)\bigl( \lcwe(\tau)\exp(\sigma - \tau) + \lcwe(\sigma) \bigr),
 & \tau \geq \sigma,\\
\end{aligned}
\end{equation}
and
\begin{equation}
\label{eq:Awtpm}
\Awtpm(s, \sigma) = \left\{
\begin{aligned}
&\exp \biggl( \pm \Ca \int_s^{\sigma} \lcwe(\nu) \, d\nu \biggr) , & s \leq \sigma, \\
&\exp \bigl( \pm \Nh(\sigma - s) \bigr), & s \geq \sigma.\\
\end{aligned}
\right.
\end{equation}
To see why, observe that 
\[
\begin{aligned}
\KK \lu(t) &= 
C_K \int_{-\infty}^{\infty} \int_{-\infty}^{\infty} E(t,\tau,\sigma) \lu(\sigma) \, d\sigma \, d\tau\\
&= 
C_K \int_{-\infty}^{\infty} \int_{-\infty}^{\infty} E(t,\tau,\sigma) \int_{-\infty}^{\infty} \Awt(s,\sigma) \f(s) \, ds \, d\sigma \, d\tau\\
&= 
C_K \int_{-\infty}^{\infty} \f(s) \int_{-\infty}^{\infty} \int_{-\infty}^{\infty} \Qw{N}{0}(e^{t-\tau}) \Awo(\tau,\sigma) \Awt(s,\sigma) \, d\sigma \, d\tau \, ds.\\
\end{aligned}
\]
The following lemma provides estimates we need.
\begin{lemma}
\label{l:est_kern_Ku}
Let\/~$s,t \in \R$. Then
\begin{equation}
\label{eq:est_kern_Ku}
\int_{-\infty}^{\infty} \int_{-\infty}^{\infty} E(t,\tau,\sigma) \Awt(s,\sigma) \,
	d\sigma \, d\tau 
\leq c \Awt(s,t),
\end{equation}
and
\begin{equation}
\label{eq:est_kern_Ku_adj}
\int_{-\infty}^{\infty} \int_{-\infty}^{\infty} E(t,\tau,\sigma) \Awtm(\sigma,s) \,
	d\sigma \, d\tau 
\leq c \Awtm(t,s),
\end{equation}
where
\[
c = 2\Biggl( \frac{3\lcwz}{\Nh} + \frac{1}{\Ca} + \frac{1}{\Cc} \Biggr)^2.
\]
\end{lemma}

\begin{proof}
First, we prove that, for~$\tau \in \R$,
\begin{equation}
\label{est:ES1}
\int_{-\infty}^{\infty} \Awo(\tau,\sigma) \Awt(s,\sigma) \, d\sigma
\leq 
2 \lcwe(\tau) \Biggl(
\frac{\lcwz}{\Nh} + \frac{1}{\Ca} + \frac{1}{\Cc} + \frac{\lcwe(\tau)}{N}
\Biggr) \Awt(s,\tau).
\end{equation}
To simplify the notation, we let
\begin{equation}
\label{eq:def_intl}
\intl(a,b) = \intlam{a}{b}, \quad a,b \in \R.
\end{equation}
Let~$s \leq \tau$. 
Then
\begin{equation}
\label{eq:aa1}
\begin{aligned}
\int_{-\infty}^s \Awo(\tau,\sigma) \Awt(s,\sigma) \, d\sigma
= {} &  
\int_{-\infty}^s e^{\Nh(\sigma - s)} \lcwe(\tau) \bigl( \lcwe(\tau)e^{\sigma - \tau} + \lcwe(\sigma) \bigr) \, d\sigma\\
\leq {} &
\lcwe(\tau)^2 \frac{e^{s - \tau}}{\Nh + 1} + \frac{\lcwz}{\Nh} \lcwe(\tau)
\leq \frac{2\lcwz}{\Nh} \lcwe(\tau)
\end{aligned}
\end{equation}
and
\begin{equation}
\label{eq:aa2}
\begin{aligned}
\int_{\tau}^{\infty} \Awo(\tau,\sigma) \Awt(s,\sigma) \, d\sigma
= {} &  
\int_{\tau}^{\infty} \lcwe^2(\tau) e^{N(\tau - \sigma)}\intl(s,\sigma)
	\, d\sigma\\
\leq {} &
\lcwe(\tau)^2 \int_{\tau}^{\infty} \frac{1}{N - \Ca\lcwe(\sigma)} \cdot \frac{\partial}{\partial \sigma} \bigl(
-e^{N(\tau - \sigma)}\intl(s,\sigma) \bigr) d\sigma \\
\leq {} & \frac{2\lcwe^2(\tau)}{N} 
\intl(s,\tau),
\end{aligned}
\end{equation}
since~$(\partial/\partial \sigma) \bigl( e^{N(\tau - \sigma)}\intl(s,\sigma) \bigr) \leq 0$ 
and~$\Ca \lcwz \leq 1/2$. Furthermore, using estimates similar to the one used in~(\ref{eq:aa2}), 
we obtain that
\[
\begin{aligned}
\int_{s}^{\tau} \Awo(\tau,\sigma) \Awt(s,\sigma) \, d\sigma
= {} &
\lcwe^2(\tau) \int_{s}^{\tau} e^{\sigma - \tau}\intl(s,\sigma) \, d\sigma
 + 
\lcwe(\tau) \int_{s}^{\tau} \lcwe(\sigma) \intl(s,\sigma) \, d\sigma\\
\leq {} & 
\lcwe^2(\tau) \bigl( \intl(s,\tau)
 - e^{s - \tau} \bigr)
 + \frac{\lcwe(\tau)}{\Ca} \bigl(
\intl(s,\tau) - 1 \bigr)\\
\leq {} &
\frac{2\lcwe(\tau)}{\Ca} \intl(s,\tau).
\end{aligned}
\]
If~$s \geq \tau$, using the same estimates as in~(\ref{eq:aa1}) and~(\ref{eq:aa2}) (where the limits~$s$ and~$\tau$
changes positions), we obtain
\[
\int_{\R \setminus [\tau,s]} \Awo(\tau,\sigma)\Awt(s,\sigma) \, d\sigma
\leq
\frac{2\lcwz\lcwe(\tau)}{\Nh} e^{\Nh(\tau - s)}
+
\frac{2\lcwe^2(\tau)}{N} e^{N(\tau - s)}.
\]
Moreover,
\[
\begin{aligned}
\int_{\tau}^s \Awo(\tau,\sigma)\Awt(s,\sigma) \, d\sigma
\leq {} &
e^{N(\tau - s)} \lcwe^2(\tau) 
	\int_{\tau}^s e^{-\Cc \lcwz(\sigma - s)} \, d\sigma
\leq \frac{\lcwe(\tau)}{\Cc} e^{\Nh(\tau - s)}.
\end{aligned}
\]
In total, we obtain the estimate in~(\ref{est:ES1}) for all~$s,\tau \in \R$.

Let us show that, for~$s,t \in \R$, 
\begin{equation}
\label{est:ES2}
\int_{-\infty}^{\infty} \lcwe(\tau) \Awt(s,\tau) \Qw{N}{0}(e^{t-\tau}) \, d\tau
\leq
\Biggl(
\frac{3\lcwz}{\Nh} + \frac{1}{\Ca} + \frac{1}{\Cc}
\Biggr) \Awt(s,t).
\end{equation}
Let~$s \leq t$. Then
\[
\int_{-\infty}^{s} \lcwe(\tau) \Awt(s,\tau) \Qw{N}{0}(e^{t-\tau}) \, d\tau
=
\int_{-\infty}^{s} 
\lcwe(\tau) e^{\Nh(\tau - s)} \, d\tau \leq \frac{\lcwz}{\Nh},
\]
and, similarly with~(\ref{eq:aa2}), 
\[
\begin{aligned}
\int_{t}^{\infty} \lcwe(\tau) \Awt(s,\tau) \Qw{N}{0}(e^{t-\tau}) \, d\tau
&= \int_t^{\infty}
\lcwe(\tau)  e^{N(t - \tau)}\intl(s,\tau) \, d\tau\\
\leq {} & 
\lcwz \int_t^{\infty}
\frac{1}{N - \Ca \lcwe(\tau)} \biggl( - \frac{\partial}{\partial \tau}
 \bigl( e^{N(t-\tau)}\intl(s,\tau)  
\bigr) \biggr) d\tau\\
\leq {} &
\frac{2\lcwz}{N} \intl(s,t).
\end{aligned}
\]
For the middle part,
\[
\int_{s}^{t} \lcwe(\tau) \Awt(s,\tau) \Qw{N}{0}(e^{t-\tau}) \, d\tau
= \int_{s}^{t} \lcwe(\tau) 
\intl(s,\tau) \, d\tau \leq \frac{1}{\Ca} \intl(s,t).
\]
If~$s \geq t$, then
\[
\int_{\R \setminus [t,s]} \lcwe(\tau) \Awt(s,\tau) \Qw{N}{0}(e^{t-\tau}) \, d\tau
\leq
\Biggl(
\frac{\lcwz}{\Nh} + \frac{2\lcwz}{N}
\Biggr) e^{\Nh(t - s)}
\]
and 
\[
\begin{aligned}
\int_{t}^s \lcwe(\tau) \Awt(s,\tau) \Qw{N}{0}(e^{t-\tau}) \, d\tau
= {} & \int_t^s \lcwe(\tau) e^{N(t-s) - \Cc\lcwz(\tau - s)} \, d\tau\\
\leq {} &
\frac{\lcwz e^{N(t-s)}}{\lcwz \Cc} \biggl( e^{-\Cc \lcwz(t - s)} - 1 \biggr)
\leq \frac{1}{\Cc} e^{\Nh(t - s)}.
\end{aligned}
\]
Thus, we obtain~(\ref{est:ES2}) for all~$s$ and~$t$ in~$\R$. It is clear that~(\ref{est:ES1}) and~(\ref{est:ES2})
imply~(\ref{eq:est_kern_Ku}) given in the lemma.

The inequality in~(\ref{eq:est_kern_Ku_adj}) can be derived analogously, using the same technique as
above.
\end{proof}


\begin{lemma}
\label{l:lu_solves}
There exists
positive constants\/~$\lcwm$,\/~$\Ca$,\/~$\Cc$, and\/~$\Cb$, depending only on\/~$N$ and\/~$p$, such that if\/~$\lcwz \leq \lcwm$
and\/~$\f(s) = \Np{\nabla f}{e^{-s}}$,\/~$s \in \R$, satisfies~{\rm(}\ref{eq:req_lu}{\rm)}, then
\[
\KK \lu(t) + \Np{\K(0)}{e^{-t}} \leq \lu(t), \quad t \in \R, \quad 1 < p < \infty.
\]
\end{lemma}

\begin{proof}
Let~$\lcwm$ be sufficiently small for Lemma~\ref{l:S_I_bound_L2} and Corollary~\ref{c:S_I_bound_Lp} to hold when~$\lcwz \leq \lcwm$.
Moreover, let~$\Ca\lcwm \leq 1/2$ and~$\Cc \lcwm \leq (N-1)/2$.
It follows from Lemma~\ref{l:est_kern_Ku} that
\[
\KK \lu(t) \leq 2C_K \Biggl( \frac{3\lcwm}{\Nh} + \frac{1}{\Ca} + \frac{1}{\Cc} \Biggr)^2 u(t).
\]
Furthermore,~$\K(0) = \psif^{-1} \Ri f$, and from Theorem~\ref{t:cont_R}, we know that
\[
\begin{aligned}
\Np{\K(0)}{e^{-t}} &\leq C_K \int_{-\infty}^{\infty} \Qw{N}{0}(e^{t-s}) \, \f(s) \, ds \\
&\leq C_K \int_{-\infty}^{\infty} \Awt(s,t) \, \f(s) \, ds 
= \frac{C_K}{\Cb} u(t).
\end{aligned}
\]
Thus,
\[
\KK \lu(t) +\Np{\K(0)}{e^{-t}} 
\leq
C_K \Biggl(  2\Biggl( \frac{3\lcwm}{\Nh} + \frac{1}{\Ca} + \frac{1}{\Cc} \Biggr)^2 + \frac{1}{\Cb} \Biggr) u(t).
\]
By choosing~$\Ca$ and~$\Cb$ large, and~$\Cc$ large enough but smaller than~$N/(2\lcwm)$, we can see that it is
possible to bound the constant by
\begin{equation}
\label{eq:krav_ci}
C_K \Biggl(  2\Biggl( \frac{3\lcwm}{\Nh} + \frac{1}{\Ca} + \frac{1}{\Cc} \Biggr)^2 + \frac{1}{\Cb} \Biggr) \leq 1.
\end{equation}
\end{proof}

\begin{lemma}
\label{l:z_solves}
Suppose that\/~$f \in \Wloc$ satisfies
\begin{equation}
\label{eq:req_f_z}
\int_{0}^{1} \! \rho^{\Nh} \Np{\nabla f}{\rho} \, \frac{d\rho}{\rho}
+ \int_1^{\infty}  \exp\biggl( \Ca \int_{1}^{\rho} \lcw(\nu) \, \frac{d\nu}{\nu} \biggr) 
\, \Np{\nabla f}{\rho} \, \frac{d\rho}{\rho} < \infty,
\end{equation}
where\/~$1 < p < \infty$ and the constants\/~$\lcwm$,\/~$\Ca$,\/~$\Cc$, and\/~$\Cb$, are given by Lemma~\ref{l:lu_solves}.
Then the function\/~$z$ defined by\/~$z(r) = \lu(\log r)$ for\/~$r > 0$ belongs to\/~$\DKK$ and
\[
\KK z(r) + \Np{\nabla f}{r} \leq z(r), \quad r > 0.
\]
\end{lemma}

\begin{proof}
Since~$f$ satisfies~(\ref{eq:req_f_z}) and~(\ref{eq:krav_ci}) holds, the function~$\lu(t)$ exists and
solves the inequality in Lemma~\ref{l:lu_solves}. This also implies that
\[
\KK z(r) \leq \lu(-\log r) < \infty, \quad r > 0.
\]
In particular,~$\KK z(1) < \infty$, which implies that~$z \in \DKK$; see~(\ref{eq:cond_domK}).
\end{proof}


\subsection{Verification of~($\mathbf{K1}$),~($\mathbf{K3}$),~($\mathbf{K4}$),
(\boldmath{$\mathscr{K}1$}), and~(\boldmath{$\mathscr{K}2$})}
\label{s:verif_rest}
Obviously~$\KK$ is linear and positive, so~\ref{i:kkpos} holds.
Furthermore,~\ref{i:kkclosed} follows from the monotone convergence theorem.
If~$\eta \in \DKK$ satisfies~$0 \leq \eta \leq z$, then by~\ref{i:kkpos} and~\ref{i:kkz},
we obtain
\[
\KK \eta(r) \leq \KK z(r) \leq z(r) - \KZ(r) \leq z(r) \text{,} \quad r > 0 \text{.}
\]
Since~$z \in \DKK$, 
this shows that~\ref{i:kkdomain} holds. 

Now, since~$(K1)$--$(K4)$ holds, there exists a minimal 
solution~$\sigma$ in~$\DKK$ to~(\ref{eq:tvsp_sigmaeq}). 
Suppose that~$u,v \in \DKS$.
By Lemma~\ref{l:ext_to_B},
\begin{equation}
\label{eq:NpKuKv2}
\Np{\K(u) - \K(v)}{r} \leq \KK(\Np{u - v}{\, \cdot \,})(r) \text{,}
\end{equation}
which is the condition in~\ref{i:DKS_est_K_KK}.
Since~$\K(u)$ is measurable and 
\[
\begin{aligned}
\Np{\K(u)}{r} &\leq \Np{\K(u) - \K(0)}{r} + \KZ(r)\\
&\leq \KK \Np{u}{\,\cdot\,}(r) + \KZ(r)\\ 
&\leq \KK \sigma (r) + \KZ(r)\\
&= \sigma(r) 
\end{aligned}
\]
for every~$r > 0$,
we see that~$\K$ maps~$\DKS$ into~$\DKS$. 

Suppose that~$\{v_k\}_{k=0}^{\infty}$ is a sequence in~$\DKS$ that 
satisfies~(\ref{eq:sumcond}). Then this is a Cauchy
sequence in~$\Lloc[p]$, so it converges to a measurable function~$v$. It follows
that~$\Np{v}{r} \leq \sigma(r)$ for~$r > 0$, so~$v \in \DKS$. Hence,~\ref{i:DKS_closed} holds.

\subsection{Existence of a Fixed Point} 
\label{s:exist_fixpt}
We now apply Theorem~\ref{t:tvsp_fixpt}, which shows that there exists a function~$u$
in~$\DKS$ such that~$\K(u) = u$. We have thus derived the following result.

\begin{lemma}
\label{l:exist_K}
Suppose that the conditions of Lemma~\ref{l:z_solves} are satisfied.
Then there exists\/~$u \in \Lloc[p]$ such that\/~$\K(u) = u$ and
\begin{equation}
\label{eq:est_sol_t}
\Np{u}{e^{-t}} \leq \lu(t) 
\text{,} \quad t \in \R \text{.}
\end{equation}
\end{lemma}

\noindent The estimate in~(\ref{eq:est_sol_t}) implies the following asymptotic behaviour of the fixed point.

\begin{lemma}
\label{l:lu_assymp}
Let\/~$u \in \Xp$ satisfy~{\rm(}\ref{eq:est_sol_t}{\rm)}. Then\/~$\Np{u}{e^{-t}}  = o(\Awtm(t,0))$ as\/~$t \rightarrow \pm \infty$.
\end{lemma}

\begin{proof}
Let~$t < -m < 0$ for some~$m > 0$. Then
\[
\frac{\Awt(\tau,t)}{\Awtm(t,0)} = \begin{cases}
\intl(\tau,0), & \tau \leq t,\\
e^{\Nh(t-\tau)} \intl(t,0), & \tau > t,
\end{cases}
\]
where~$\intl$ is given by~(\ref{eq:def_intl}), and thus, for every~$\epsilon > 0$, we can choose~$m$ large enough so that for~$t < -m$,
\[
\begin{aligned}
\frac{\lu(t)}{\Awtm(t,0)} \leq {} & \int_{-\infty}^{t} \intl(\tau,0) \, \f(\tau) \, d\tau
+ \int_{t}^{-m} e^{\Nh(t-\tau)} \intl(t,0) \f(\tau) \, d\tau \\
& + \int_{-m}^{\infty} e^{\Nh t} \intl(t,0) e^{-\Nh \tau} \f(\tau) \, d\tau\\
\leq {} & \frac{\epsilon}{3} + \int_{-\infty}^{-m} e^{\Nh(t-\tau)} \intl(t,\tau) \intl(\tau,0) \f(\tau) \, d\tau
+ e^{-\beta t} \int_{-m}^{\infty} e^{-\Nh \tau} \f(\tau) \, d\tau\\
\leq {} & \frac{2\epsilon}{3} + \int_{-\infty}^{-m} e^{-\beta(t-\tau)} \intl(\tau,0) \f(\tau) \, d\tau
 \leq \epsilon,
\end{aligned}
\]
where~$\beta = N - \Ca \lcwz - \Cc \lcwz > 0$ if~$\Ca\lcwz \leq 1/2$ and~$\Cc\lcwz \leq (N-1)/2$. 
Similarly, one can show that if~$t > 0$,
we obtain that~$\lu(t)/\Awtm(t,0) \leq \epsilon$ if~$t > m$.
\end{proof}

\subsection{Uniqueness of Fixed Points}
\label{s:uniq}
We can now prove the following uniqueness result.

\begin{lemma}
\label{t:uniq_K}
Suppose that the conditions in Lemma~\ref{l:z_solves}
are satisfied. 
Then there exists at most one solution\/~$u$ in\/~$\Xp$ to the equation\/~$\K(u) = u$ such that
\begin{equation}
\label{eq:cond_uniq}
\Np{u}{e^{-t}} = O(\Awtm(t,0)), \quad t \rightarrow \pm \infty.
\end{equation}
\end{lemma}

\begin{proof}
To simplify notation, let
\[
\pa[t]{u} = \Np{u}{e^{-t}}, \quad t \in \R.
\] 
Choose~$\DK$ as the linear space of functions~$u \in \Lloc[p]$ such
that~(\ref{eq:cond_uniq}) holds.
Suppose that~$u \in \DK$. Then there exists constants~$A' > 0$ and~$m > 0$ such that
\[
\pa[t]{u} \leq A' \Awtm(t,0), \quad |t| \geq m.
\]
Let the constant~$A''$ be given by
\[
A'' = \sup_{|t| \leq m} \Awt(t,0) \pa[t]{u}.
\]
By continuity, it is clear that~$A'' < \infty$. Define~$A = \max\{\,A',A''\,\}$. We find that
\begin{equation}
\label{eq:uniq_K1}
\begin{aligned}
\KK \pa[ \, \cdot \, ]{u}(t) &= 
C_K \int_{-\infty}^{\infty} \int_{-\infty}^{\infty} E(t,\tau,\sigma) 
	\pa[\sigma]{u} \, d\sigma \, d\tau\\
&\leq
A C_K 
\int_{-\infty}^{\infty} \int_{-\infty}^{\infty} E(t,\tau,\sigma) 
	\Awtm(\sigma,0) \, d\sigma \, d\tau\\
&\leq
A C_K c \Awtm(t,0)
\end{aligned}
\end{equation}
for~$t \in \R$. The last inequality follows from~(\ref{eq:est_kern_Ku_adj}) in
Lemma~\ref{l:est_kern_Ku}, where the constant~$c$
is also defined. Inequality~(\ref{eq:uniq_K1}) implies that if~$u$ belongs to~$\DK$,
then~$\pa[\, \cdot \,]{u}$ belongs to~$\DKK$, and thus,~$\K(u)$ is defined. This inequality
also implies that~$\K$ maps~$\DK$ into~$\DK$.
To apply Theorem~\ref{t:tvsp_fixptu},
we need to verify that~\ref{i:KKnu}--\ref{i:DK_est_K_KK} hold.
Inequality~(\ref{eq:uniq_K1}) shows that~\ref{i:KKfp} holds:
if~$u$ belongs to~$\DK$ and~$\eta$ belongs to~$\DKK$ such that~$0 \leq \eta(t) \leq \pa[t]{u}$ for~$t \in \R$, 
then the monotonicity of~$\KK$ implies that~$\KK \eta(t) = O(\Awtm(t,0))$ 
as~$t \rightarrow \pm \infty$. Hence,~$\KK\eta$ belongs to~$\DKK$.
The fact that~\ref{i:DK_est_K_KK} holds follows from Lemma~\ref{l:ext_to_B}. 

Furthermore, by applying~(\ref{eq:uniq_K1})~$n$ times, we obtain that
\[
\KK^n \pa[\, \cdot \,]{u}(t) \leq (C_K c)^n A \Awtm(t,0).
\]
Since~$C_K c < 1$, which follows from~(\ref{eq:krav_ci}), and~$\Awtm(t,0) < \infty$, 
\[
\KK^n \pa[\, \cdot \,]{u}(t) \rightarrow 0, \quad \text{as } n \rightarrow \infty.
\]
Hence, Theorem~\ref{t:tvsp_fixptu} implies that the fixed point is indeed unique.
\end{proof}

\begin{remark}
\label{r:soluniq}
The condition in~{\rm(}\ref{eq:cond_uniq}{\rm)} is a natural condition if we consider the 
solution\/~$u \in \DK$ in Lemma~\ref{l:exist_K}, which satisfies\/~$\Np{u}{e^{-t}} \leq \lu(t)$,
so by Lemma~\ref{l:lu_assymp},~{\rm(}\ref{eq:cond_uniq}{\rm)} is valid.
\end{remark}


\section{Proof of the Main Results}
In the previous section, we proved that there exists a fixed point of~$\K$, and that
it is unique under certain conditions on~$f$. 
We will now use these theorems and results from Section~\ref{s:rieszpot} to prove similar results for solutions to~(\ref{eq:maineq}).

\subsection{Existence of Solutions to~(\ref{eq:maineq})} 
We start with deriving two technical lemmas which will show that we can use
Theorem~\ref{t:inv} to recover~$(\Ipsi - \S)u + f$ from the equation~$\K(u) = u$.
First we find integrated estimates for~$\Awt$, where~$\Awt$ is defined by~(\ref{eq:Awtpm}).

\begin{lemma}
\label{l:sol_in_Xp_estimates}
Suppose that\/~$\Ca \lcwz \leq 1/2$ and\/~$\Cc \lcwz \leq (N-1)/2$.
Then, 
\begin{equation}
\label{eq:sol_in_Xp_cond}
\int_{-\infty}^{\infty} \Qw{N}{1}(e^{-t}) \Awt(s,t) dt
\leq
\Clcwzs \, \Qw{\Nh}{1} (e^{-s})
\end{equation}
and 
\begin{equation}
\label{eq:sol_in_Yonez_cond}
\int_{-\infty}^{0} e^{-t} \, \Awt(s,t) \, dt 
+
\int_0^{\infty} (1+t)e^{-Nt} \, \Awt(s,t) \, dt 
\leq \Clcwzs \, 
\Qw{\Nh}{1}(e^{-s})
\end{equation}
for\/~$s \in \R$.
The function\/~$\Clcwzs$ will depend on\/~$\lcwz$ and\/~$s$, but for fixed\/~$\lcwz$ it is uniformly bounded
with respect to\/~$s$, and 
\begin{equation}
\label{eq:lcwz_zero}
\limsup_{\lcwz \rightarrow 0} \Clcwzs \leq \biggl\{ \begin{array}{ll} C, & s \leq 0, \\ C(1+s), & s > 0, \end{array}
\end{equation}
where\/~$C$ only depends on\/~$N$ and\/~$p$.
\end{lemma}

\begin{proof}
Using notation from Section~\ref{s:aux_eq}, we 
show~(\ref{eq:sol_in_Xp_cond}) first. 
Proceeding similarly with the proof of Lemma~\ref{l:est_kern_Ku}, we consider two cases.

First, let~$s \leq 0$. The left-hand side in~(\ref{eq:sol_in_Xp_cond}) is given by
\begin{equation}
\label{eq:sol_in_Xp_cond_sm}
\begin{aligned}
&\int_{-\infty}^s e^{-t} e^{\Nh(t-s)} \, dt + \int_s^0 e^{-t} \intl(s,t) \, dt
+ \int_0^{\infty} e^{-Nt} D(s,t) \, dt\\ 
& \qquad \leq 
\frac{e^{-s}}{\Nh - 1} + 2 e^{-s} + \frac{2}{N} D(s,0)
\leq  C e^{-s},
\end{aligned}
\end{equation}
where we exploited that
\[
\Ca \int_s^0 \lcwe(\nu) \, d\nu \leq - \Ca \lcwz s \leq -s, \qquad s < 0,
\]
so~$D(s,0) \leq e^{-s}$. The constant~$C$ in~(\ref{eq:sol_in_Xp_cond_sm}) only depends on~$N$ and~$p$.

Let~$s \geq 0$. In the same manner as above,  
the left-hand side in~(\ref{eq:sol_in_Xp_cond}) is bounded by
\[
\begin{aligned}
&\int_{-\infty}^0 e^{-t} e^{\Nh(t-s)} \, dt + \int_0^s e^{-Nt} e^{\Nh(t-s)} \, dt
+ \int_s^{\infty} e^{-Nt} D(s,t) \, dt\\ 
& \qquad \leq 
\frac{e^{-\Nh s}}{\Nh - 1} + \frac{e^{-\Nh s} - e^{-Ns}}{N - \Nh} + \frac{2}{N} e^{-Ns} 
\leq \left( \frac{4}{N-1} + \frac{1 - e^{-\Cc\lcwz s}}{\Cc\lcwz} \right) e^{-\Nh s},
\end{aligned}
\]
which completes the proof of~(\ref{eq:sol_in_Xp_cond}) since
\[
\limsup_{\lcwz \rightarrow 0} \left( \frac{4}{N-1} + \frac{1 - e^{-\Cc\lcwz s}}{\Cc\lcwz} \right)
\leq C(1+s). 
\]

To prove~(\ref{eq:sol_in_Yonez_cond}), we proceed similarly.
Let~$s \leq 0$. Then the left-hand side in~(\ref{eq:sol_in_Yonez_cond}) is given by
\[
\int_{-\infty}^{s} e^{-t} e^{\Nh(t-s)} \, dt + \int_s^{0} e^{-t} \intl(s,t) \, dt 
+ \int_0^{\infty} (1+t)e^{-Nt}D(s,t) \, dt.
\]
The first two terms can be estimated in the same manner as~(\ref{eq:sol_in_Xp_cond_sm}): 
\[
\int_{-\infty}^{s} e^{-t} e^{\Nh(t-s)} \, dt + \int_s^{0} e^{-t} \intl(s,t) \, dt \\
\leq
\frac{e^{-s}}{\Nh - 1} + 2 e^{-s}.
\]
To investigate the third term, we use integration by parts:
\begin{equation}
\label{eq:kpp1_pi1}
\begin{aligned}
\int_0^{\infty} (1+t)e^{-Nt}D(s,t) \, dt
&=
\int_0^{\infty} (1+t) \frac{1}{N - \Ca \lcwe(t)} \biggl( - \frac{\partial}{\partial t} e^{-Nt} \intl(s,t) \biggr) \, dt\\
&\leq
\frac{2}{N} \int_0^{\infty} (1+t) \biggl( - \frac{\partial}{\partial t} e^{-Nt} \intl(s,t) \biggr) \, dt\\
&\leq
\frac{2}{N} \intl(s,0) +
\frac{4}{N^2} \int_0^{\infty} \biggl( - \frac{\partial}{\partial t} e^{-Nt} \intl(s,t) \biggr) \, dt\\
&\leq C \intl(s,0),
\end{aligned}
\end{equation}
where~$C$ only depends on~$N$ and~$p$.
Since~$\intl(s,0) \leq e^{-s/2} \leq e^{-s}$, we obtain that~(\ref{eq:sol_in_Yonez_cond}) holds for~$s \leq 0$. 

Suppose that~$s \geq 0$. We proceed analogously with the case when~$s \leq 0$. The left-hand side of~(\ref{eq:sol_in_Yonez_cond})
is given by
\[
\int_{-\infty}^{0} e^{-t} e^{\Nh(t-s)} \, dt + \int_{0}^s (1+t) e^{-Nt} e^{\Nh(t-s)} \, dt 
+ \int_s^{\infty} (1+t)e^{-Nt}D(s,t) \, dt.
\]
The first term can be calculated as
\[
\int_{-\infty}^{0} e^{-t} e^{\Nh(t-s)} \, dt  = \frac{e^{-\Nh s}}{\Nh - 1}.
\]
The other terms can be bounded in the same manner as~(\ref{eq:kpp1_pi1}) above using integration by parts:
\[
\begin{aligned}
\int_{0}^s (1+t) e^{-Nt} e^{\Nh(t-s)} \, dt &= e^{-\Nh s}  \int_0^s (1+t)e^{-\Cc \lcwz t} \, dt \\
&= \left( \frac{1 - (1+s)e^{-\Cc \lcwz s}}{\Cc \lcwz} + \frac{1 - e^{-\Cc\lcwz s}}{\Cc^2\lcwz^2} \right) e^{-\Nh s}
\end{aligned}
\]
and
\[
\int_s^{\infty} (1+t)e^{-Nt}D(s,t) \, dt \leq \frac{2}{N}(1+s)e^{-Ns} + \frac{4}{N^2} e^{-Ns}.
\]
Since
\[
\limsup_{\lcwz \rightarrow 0} 
	\left( \frac{1 - (1+s)e^{-\Cc \lcwz s}}{\Cc \lcwz} + \frac{1 - e^{-\Cc\lcwz s}}{\Cc^2\lcwz^2} \right) 
	\leq C(1+s), \quad s > 0,
\]
and
\[
(1+s) e^{-Ns} \leq \frac{C}{\Cc \lcwz} e^{-\Nh s}, \quad s > 0,
\]
it follows that~(\ref{eq:sol_in_Yonez_cond}) and~(\ref{eq:lcwz_zero}) hold for~$s \geq 0$ as well.
\end{proof}

\begin{lemma}
\label{l:sol_in_Xp}
Let\/~$u \in \Lloc[p]$ satisfy\/~$\Np{u}{e^{-t}} \leq C \, \lu(t)$ for\/~$t \in \R$
and suppose that\/~$\Ca \lcwz \leq 1/2$ and\/~$\Cc \lcwz \leq (N-1)/2$.
If\/~$f \in \Wloc$ satisfies
\begin{equation}
\label{eq:cond_u_Xp_simp}
\int_0^{\infty} \Qw{\Nh}{1} (\rho) \, \Np{\nabla f}{\rho}
		\frac{d\rho}{\rho} < \infty \text{,}
\end{equation}
then\/~$u \in \Xp$ and\/  
$(\Ipsi -\S)u \in \Yonez$.
\end{lemma}

\begin{proof}
Using notation from Section~\ref{s:aux_eq}, we see that
\begin{equation}
\begin{aligned}
\| u \|_{\Xp} &= \int_0^{\infty} \Qw{N}{1}(\rho) \, \Np{u}{\rho} \, \frac{d\rho}{\rho}
= \int_{-\infty}^{\infty} \Qw{N}{1}(e^{-t}) \, \Np{u}{e^{-t}} \, dt\\
&\leq C \int_{-\infty}^{\infty} \Qw{N}{1}(e^{-t}) 
	\int_{-\infty}^{\infty} \Awt(s,t) \f(s) ds \, dt\\
&= C \int_{-\infty}^{\infty} \f(s) 
	\int_{-\infty}^{\infty} \Qw{N}{1}(e^{-t}) \Awt(s,t) dt \, ds.
\end{aligned}
\end{equation}
Inequality~(\ref{eq:sol_in_Xp_cond}) now implies that~$u \in \Xp$.

Turning our attention to the second part of the lemma, i.e., that~$(\Ipsi - \S) u$ belongs to~$\Yonez$, we 
need to prove two things:
\begin{equation}
\label{eq:cond_ipsi_yonez}
\int_0^{\infty} H(\rho) \, \Np{\nabla (\Ipsi - \S)u}{\rho} \, \frac{d\rho}{\rho} < \infty
\end{equation}
and
\[
h_R(r) \rightarrow 0, \quad \text{as } r \rightarrow \infty,
\]
where~$H(\rho) = \rho^N(1 - \log \rho)$ if~$0 < \rho \leq 1$ and~$H(\rho) = \rho$ if~$\rho \geq 1$,
and the function~$h(x) = (\Ipsi - \S)u(x)$ for~$x \in \R^N$.

We know that~$\partial_k (\Ipsi - \S)u$,~$k=1,2,\ldots,N$, are defined and the 
representation in~(\ref{eq:diff_S_I_1u})
holds. Moreover, by Lemma~\ref{l:ext_to_B} it is possible to 
extend~(\ref{eq:est_diff_I_1_S}) for all~$u$ in~$\Xp$. 
Hence, since~$\Np{u}{e^{-t}} \leq C \lu(t)$ and~$\KK \lu \leq \lu$,
\[
\Np{\nabla(\Ipsi - \S)u}{r} \leq \KK(\Np{u}{\, \cdot \,})(r) \leq 
C \KK \lu (r) \leq C \lu(r), \quad r > 0,
\]
which implies that
\[
\begin{aligned}
\int_0^{\infty} H(r) \, \Np{\nabla (\Ipsi - \S)u}{r} \, \frac{dr}{r} &\leq
C \, \int_{-\infty}^{\infty} H(e^{-t}) \int_{-\infty}^{\infty} \f(s) \, \Awt(s,t) \, ds \, dt\\
&\leq
C \, \int_{-\infty}^{\infty} \f(s) \int_{-\infty}^{\infty} H(e^{-t}) \, \Awt(s,t) \, dt \, ds.\\
\end{aligned}
\]
Since~(\ref{eq:sol_in_Yonez_cond}) holds, it follows that~(\ref{eq:cond_ipsi_yonez}) is valid.

We will now verify that
\begin{equation}
\label{eq:lim_hr}
\lim_{r \rightarrow \infty} h_R(r) = \lim_{r \rightarrow \infty} \int_{S^{N-1}} h(r\theta) \, d\theta = 0 \text{,}
\end{equation}
where~$h(x) = (\Ipsi - \S)u(x)$,~$x \in \R^N$. 
By Lemma~2.4 in~\cite{thim1},
the fact that
\[
\int_{1}^{\infty} \Np{\nabla h}{\rho} \, d\rho < \infty 
\]
implies that the limit in~(\ref{eq:lim_hr}) exists.
We now obtain that
\[
\begin{aligned}
\inf_{r \leq \rho < 2r} |h_R(\rho)| &\leq \frac{1}{r} \int_r^{2r} |h_R(\rho)| \, d\rho\\
&\leq \frac{1}{r^N} \int_r^{2r} \rho^{N-1} \int_{S^{N-1}} |h(\rho\theta)| \, d\theta \, d\rho\\
& \leq C \, r^{-N/p} \biggl( 
        \int_{r \leq |x| < 2r} |h(x)|^p \, dx 
 \biggr)^{1/p}\\
&= C \, \Np{h}{r} \text{.}
\end{aligned}
\]
Obviously~$|h(x)| \leq C \, \Ipsi |u|(x)$, so~(\ref{eq:Np_est_I_1}) in Theorem~\ref{t:cont_I1} implies that
\[
\begin{aligned}
\Np{h}{r} \leq C \biggl( 
&r^{1-N} \int_0^1 \rho^{N} \Np{u}{\rho} \, \frac{d\rho}{\rho}
+ r^{1-N} \int_1^r \rho^{N} \Np{u}{\rho} \, \frac{d\rho}{\rho}\\
&+ \int_r^{\infty} \Np{u}{\rho} \, d\rho
\biggr) 
\end{aligned}
\]
for~$r > 1$. Since~$u \in \Xp$,  
the first and last term in the right-hand side tend to zero as~$r \rightarrow \infty$. To show
that this is also true for the middle term, let~$m > 1$ be fixed. Then
\[
\limsup_{r \rightarrow \infty} r^{1-N} \int_1^r \rho^{N} \, \Np{u}{\rho} \, \frac{d\rho}{\rho}
\leq \int_m^{\infty} \Np{u}{\rho} \, d\rho \text{.}
\]
The number~$m$ is arbitrary, so the limit must be zero 
since~$u \in \Xp$.
Thus,~$h_R(\rho_n) \rightarrow 0$ for some sequence~$\rho_n$ such that~$\rho_n \rightarrow \infty$.
Hence~(\ref{eq:lim_hr}) must hold, so~$(\Ipsi - \S)u \in \Yonez$.
\end{proof}

\subsection{Proof of Theorem~\ref{t:i:exist}}
\label{s:proof_exist}
%
If~$f \in \Yonew$, then the solution~$z$ to~(\ref{eq:z})
exists.
Applying Lemma~\ref{l:exist_K}, we obtain a fixed point~$u \in \Lloc[p]$ of~$\K$ such
that
\[
\Np{u}{r} \leq z(r), \quad  r > 0.
\]
Lemma~\ref{l:sol_in_Xp} shows that~$u \in \Xp$
and that the function~$(\Ipsi - \S)u$ belongs to~$\Yonez$. 
Clearly,~$f \in \Yonew$ implies that~$f \in \Yonez$,
so Theorem~\ref{t:inv} proves that
\[
\Ipsi u = \Ipsi \K(u) = \Ipsi \Ri ((\Ipsi - \S)u + f) = \Ipsi u - \S u + f \text{,}
\]
or equivalently, that~$\S u = f$. 

\begin{corollary}
\label{c:cont_req_sol}
If\/~$\lcwz \rightarrow 0$, then the condition that\/~$f \in \Yonew$ reduces to\/~$f \in \Yonez$. 
\end{corollary}

\noindent In other words, we recover the authors' previous result (see Theorem~\ref{t:inv}).

\begin{proof}
Letting~$\lcwz \rightarrow 0$ in~(\ref{eq:sol_in_Xp_cond}) and~(\ref{eq:sol_in_Yonez_cond})
and using~(\ref{eq:lcwz_zero}) shows that the right-hand sides of~(\ref{eq:sol_in_Xp_cond})
and~(\ref{eq:sol_in_Yonez_cond}) tend to~$C (1+s) \Qw{N}{1}(e^{-s})$ if~$s > 0$
and~$C \Qw{N}{1}(e^{-s})$ if~$s \leq 0$. 
The condition in~(\ref{eq:cond_u_Xp_simp}) in Lemma~\ref{l:sol_in_Xp} can now be replaced by
\[
\int_0^1 (1-\log \rho)\rho^N \Np{\nabla f}{\rho} \, \frac{d\rho}{\rho}
+ \int_1^{\infty} \Np{\nabla f}{\rho} \, d\rho < \infty.
\]
\end{proof}

\subsection{Uniqueness of Solutions: Proof of Theorem~\ref{t:i:uniq}}
\label{s:proof_uniq}

%
We show that~$\S u = 0$ implies that~$\K(u) = u$ and use Theorem~\ref{t:uniq_K} to
deduce that~$u = 0$.

Let~$\Bs$ be the Banach space introduced in the proof of Lemma~\ref{l:ext_to_B} and pick a 
sequence~$u_n \in \Ckc$ such that~$u_n \rightarrow u$ in~$\Bs$. It is clear
that~$u_n \rightarrow u$ in~$\Lloc[p]$. 
The definition of~$\K$ implies that
\[
\K(u_n) = \Ri (\Ie - \S)u_n \rightarrow \K(u)
\]
in~$\Lloc[p]$. 
It is straightforward to verify that~$\Ri \Ie u_n = u_n$, for example by means of the Fourier transform. 
As in Lemma~\ref{l:ext_to_B}, we extend the operator~$\Ri \S$ to~$\Bs$ so that
\begin{equation}
\label{eq:est_RS}
\Np{\Ri \S u}{r} \leq C \int_0^{\infty} \int_0^{\infty}  E(r,\rho,\sigma) \, \Np{u}{\sigma} \, \frac{d\sigma}{\sigma} \, \frac{d\rho}{\rho}
\end{equation}
for every~$u \in \Bs$. Clearly~$\Ri \S u = 0$ if~$u \in \Bs$ satisfies~$\S u = 0$. 
From~(\ref{eq:est_RS}), we also
obtain that~$\Ri \S u_n \rightarrow \Ri \S u$ in~$\Lloc[p]$.
Hence
\[
\K(u_n) = \Ri (\Ie - \S) u_n = \Ri \Ie u_n - \Ri \S u_n = u_n - \Ri \S u_n \rightarrow u - \Ri \S u = u
\]
in~$\Lloc[p]$. By uniqueness of limits, we have~$\K(u) = u$.

Now, Theorem~\ref{t:uniq_K} states that the solutions to~$\K(u) = u$ that 
satisfy~(\ref{eq:cond_uniq}) are unique. Since~$u = 0$ is one such solution,
we must conclude that~$u = 0$ is the only possibility.


\newcommand{\coms}{\ensuremath{\left[ \S, \eta \right] \! u}}
\newcommand{\comspd}{\ensuremath{\partial_k \left[ \S, \eta \right] \! u}}
\newcommand{\comsd}{\ensuremath{\left[ \partial_k \S, \eta \right] \! u}}
\newcommand{\comsdd}{\ensuremath{\left[ \nabla \S, \eta \right] \! u}}

\newcommand{\comco}{\ensuremath{\chi}}
\newcommand{\comcomp}{\ensuremath{\Psi}}

\section{Local Estimates}
\label{s:assymp}
Let~$r_0$ be positive and let~$\eta \in C^{\infty}(\R)$ be a 
cut-off function such that~$\eta(r) = 1$ if~$r < r_0$ and~$\eta(r) = 0$ if~$r > 2r_0$.
Let~$u \in \Xp$,~$1 < p < \infty$, solve the equation~$\S u(x) = f(x)$ for~$|x| \leq 2r_0$, where~$f \in \Wloc$.
Furthermore, let~$\comco \in C_c^{\infty}(\R)$ satisfy~$\comco(r) = 0$ if~$r < r_0$ or~$r > 2r_0$. We will require that
\[
\int_{0}^{\infty} \comco(\rho) \, \frac{d\rho}{\rho} = \frac{N}{|S^{N-1}|},
\]
where~$|S^{N-1}|$ is the surface measure of the unit sphere in~$\R^N$. 
We define~$\comcomp(y)$ for~$y \in \R^N \setminus \{0\}$ by
\[
\comcomp(y) = - \comco(|y|) \sum_{i=1}^N \gamma_i \frac{y_i}{|y|}, \quad \text{where }  
\gamma_i = \int_{\R^N} \frac{(1-\eta(y))u(y)}{|y|^N} \frac{y_k}{|y|} \, dS(y). 
\]
Here,~$dS(y) = \sqrt{1 + |\nabla \vp(y)|^2} \, dy$.
For~$k=1,2,\ldots,N$, the function~$\comcomp$ satisfies
\begin{equation}
\label{eq:comcomp_kill}
\begin{aligned}
\int_{\R^N} \frac{\comcomp(y) y_k \, dS(y)}{|y|^{N+1}} &= 
	- \sum_{i=1}^N \gamma_i \int_0^{\infty} \comco(\rho)\int_{S^{N-1}} \theta_i \theta_k \, d\theta \, \frac{d\rho}{\rho}
=
-\gamma_k. 
\end{aligned}
\end{equation}
We multiply the equation~$\S u = f$ by~$\eta(|x|)$, and add~$\S (\eta u)$ and~$\S \comcomp$
to both sides and rearrange:
\begin{equation}
\label{eq:Suf_co}
\S (\eta u + \comcomp) = \eta f + \coms + \S \comcomp,
\end{equation}
where~$\coms(x) = \S (\eta u) (x) - \eta(x) \S u (x)$,~$x \in \R^N$. We wish to estimate the
gradient of the right-hand side of~(\ref{eq:Suf_co}).

\begin{lemma}
\label{l:est_Np_comm}
Let\/~$u \in \Xp$,\/~$1 < p < \infty$, and suppose that\/~$r_0 > 0$. Then
\[
\Np{\nabla (\coms + \S \comcomp) }{r} \leq \Biggl\{ 
\begin{aligned}
&C(r + \lcw(r)) \|u\|_{\Xp} , & r \leq r_0,\\
&C r^{-N} \, \|u \|_{\Xp}, & r > r_0,
\end{aligned}
\]
where\/~$C$ depends on\/~$N$,\/~$p$, and\/~$r_0$.
\end{lemma}

\begin{proof}
Let~$0 < r < r_0/2$ and~$r \leq |x| < 2r$. Then~$\comspd = \partial_k \S ((\eta - 1)u)$, and from the representation in~(\ref{eq:diff_Su}) it follows that
\[
\begin{aligned}
\frac{\partial_k \bigl( \left[ \S, \eta \right] \! u + \S \comcomp \bigr)(x)}{1-N}  
= {} & T_k\bigl((1 - \eta)u + \comcomp \bigr)(x) + \partial_k \vp (x) T_{N+1}\bigl( (1-\eta)u + \comcomp \bigr)(x)\\
= {} &
\int_{\R^N} \frac{(x_k - y_k)\bigl( (1 - \eta(y))u(y) + \comcomp(y) \bigr) \, dy}{|\ld[x] - \ld[y]|^{N+1}}\\
& + \partial_k \vp (x) \int_{\R^N} \frac{(\vp(x) - \vp(y))\bigl( (1 - \eta(y))u(y) + \comcomp(y) \bigr) \, dy}{|\ld[x] - \ld[y]|^{N+1}}. 
\end{aligned}
\]
The second term in the right-hand side can be estimated by
\[
C \lcw(r) \int_{|y| > r_0} \frac{  |u(y)| + |\comcomp(y)|}{|y|^N} \, dy
\leq
C \lcw(r)  \int_{|y| > r_0} \frac{|u(y)|}{|y|^N} \, dy 
\]
since
\[
|\comcomp(y)| \leq |\comco(|y|)|  
\sum_{i=1}^N |\gamma_i| \leq C|\comco(|y|)| \int_{|z| > r_0} \frac{|u(z)| \, dz}{|z|^N},  \quad y \in \R^N.
\]
Let~$\tau = |x|/|y|$. We assume that~$\tau \leq 1/2$. The kernel in~$T_k$ can be expressed by
\[
\frac{x_k - y_k}{|\ld[x] - \ld[y]|^{N+1}} = \frac{1}{|y|^N} g(\tau),
\quad \text{where }
g(\tau) = \frac{\frac{x_k}{|x|} \tau - \frac{y_k}{|y|}}{\left|\frac{\ld[x]}{|x|}\tau - \frac{\ld[y]}{|y|} \right|^{N+1}}.
\]
We see that~$g(0) = -(y_k/|y|)(1+\Lx^2)^{-(N+1)/2}$, and also that~$g'(\tau)$ is uniformly 
bounded,~$|g'(\tau)| \leq C$, for~$\tau \leq 1/2$. The constant depends only on~$N$.
Thus,~$|g(\tau) - g(0)| \leq C \tau$ for~$\tau \leq 1/2$. Hence,
\[
\begin{aligned}
|T_k((1-\eta)u + \comcomp)(x)| \leq {} & \frac{1}{(1+\Lx^2)^{\frac{N+1}{2}}} \biggl| \int_{\R^N} \frac{y_k}{|y|} \cdot 
	\frac{(1-\eta(y))u(y) + \comcomp(y)}{|y|^N} \, dS(y) \biggr| \\
& + 
C |x| \int_{|y| > r_0} \frac{|u(y)| + |\comcomp(y)|}{|y|^{N+1}} \, dy.
\end{aligned}
\]
It is clear from~(\ref{eq:comcomp_kill}) that
\[
\int_{\R^N} \frac{y_k}{|y|} \cdot \frac{(1-\eta(y))u(y) + \comcomp(y)}{|y|^N} \, dS(y) = 0, \quad k=1,2,\ldots,N.
\]
Moreover, 
\[
|x|\int_{|y| > r_0} \frac{|u(y)| + |\comcomp(y)|}{|y|^{N+1}} \, dy
\leq
C r \int_{|y| > r_0} \frac{|u(y)|}{|y|^{N}} dy,
\] 
so
\[
\Np{T_k((1-\eta)u + \comcomp)}{r} \leq C \bigl( r + \lcw(r) \bigr) \int_{|y| > r_0} \frac{|u(y)| }{|y|^N} dy.
\]
The right-hand side is finite since~$u \in \Xp$. Thus, we obtain that
\begin{equation}
\label{eq:est_Np_comm_m}
\Np{\nabla \bigl( \left[ \S, \eta \right] \! u + \S \comcomp \bigr)}{r} \leq C \bigl( r + \lcw(r) \bigr) \|u\|_{\Xp},
\quad 0 < r < \frac{r_0}{2}.
\end{equation}
Let~$r > 4r_0$. Then~$\comspd = \partial_k \S(\eta u)$, and
\[
|\comspd(x) | \leq \frac{C}{|x|^N} \int_{|y| < 2r_0} |u(y)| \, dy \leq \frac{C}{|x|^{N}} \|u\|_{\Xp}.
\]
Furthermore,
\[
\begin{aligned}
|\partial_k \S \comcomp (x)| &\leq C |T_k \comcomp (x)| + C |T_{N+1} \comcomp(x)| \\
& \leq \frac{C}{|x|^N} \int_{r_0 \leq |y| \leq 2r_0} |\comcomp(y)| \, dy
\leq \frac{C}{|x|^N} \int_{|y| > r_0} \frac{|u(y)|}{|y|^N} dy. 
\end{aligned}
\]
This implies that
\begin{equation}
\label{eq:est_Np_comm_p} 
\Np{\nabla \bigl( \left[ \S, \eta \right] \! u + \S \comcomp \bigr)}{r} 
\leq C r^{-N} \|u\|_{\Xp}, \quad r > 4r_0.
\end{equation}
If~$r_0/2 \leq r \leq 4r_0$, then
\[
\begin{aligned}
\comspd(x) &= \partial_k \S \eta u(x) - \eta'(|x|) \frac{x_k}{|x|} \S u (x) - \eta(|x|) \partial_k \S u(x)\\
&= \comsd(x) - \eta'(|x|)\frac{x_k}{|x|} \S u(x).
\end{aligned}
\]
From the representation in~(\ref{eq:diff_Su}), we obtain that
\[
\begin{aligned}
|\comsd(x)| &\leq C \int_{\R^N} \frac{|\eta(x) - \eta(y)|}{|\ld[x] - \ld[y]|^N} \, |u(y)| \, dy 
\leq 
C \int_{\R^N} \frac{|u(y)|}{|x-y|^{N-1}} \, dy.\\
\end{aligned}
\]
Thus,~$|\comsd(x)| \leq C \Ie|u|(x)$, so
\[
\begin{aligned}
\Np{\comsdd}{r} 
& \leq  C r \, \int_0^{\infty} \Qw{N}{1} \left( \frac{\rho}{r} \right) \, \Np{u}{\rho} \, \frac{d\rho}{\rho}
 \leq C \| u \|_{\Xp}.
\end{aligned}
\]
Similarly, since~$|\eta' \S u| \leq |\eta'| \Ie|u|$, we obtain that
$
\Np{\eta' \S u}{r} \leq C \, \|u\|_{\Xp} 
$
for~$r_0/2 \leq r \leq 2r_0$, and that~$\eta' \S u = 0$ otherwise.
Since also
\[
\begin{aligned}
\Np{\nabla \S \comcomp}{r} &\leq C \int_{0}^{\infty} \Qw{N}{0} \left( \frac{\rho}{r} \right) \, \Np{\comcomp}{\rho} \, \frac{d\rho}{\rho}
 \leq C \int_{|y|>r_0} \frac{|u(y)|}{|y|^N} \, dy ,
\end{aligned}
\]
it is clear that
\begin{equation}
\label{eq:est_Np_comm_c}
\Np{\nabla \bigl( \left[ \S, \eta \right] \! u + \S \comcomp \bigr)}{r} 
\leq C \|u\|_{\Xp}, \quad \frac{r_0}{2} \leq r \leq 4r_0.
\end{equation}
The desired result now follows from~(\ref{eq:est_Np_comm_m}),~(\ref{eq:est_Np_comm_p}), and~(\ref{eq:est_Np_comm_c}).
\end{proof}

\subsection{Proof of Theorem~\ref{t:i:local}}
Lemma~\ref{l:est_Np_comm} implies, for~$u$ in~$\Xp$, that~$\coms$ belongs to~$\Yonew$ and 
that the radial part~$(\coms)_R$ tends to zero as~$r \rightarrow \infty$. Similarly, the same is true for~$\S \comcomp$.
Moreover, since~$f$ satisfies~(\ref{eq:req_local}),~$\eta f \in \Yonew$ and~$(\eta f)_R \rightarrow 0$.
Thus, by Theorem~\ref{t:i:exist},
\begin{equation}
\label{eq:Sw}
\S w = \eta f + \left[ \S, \eta \right] \! u + \S \comcomp
\end{equation}
has a solution~$w \in \Xp$ which satisfies
\begin{equation}
\label{eq:est_w}
\begin{aligned}
\Np{w}{r} \leq {} & \Cb \int_0^{r} \left( \frac{\rho}{r} \right)^{\Nh} \Np{\nabla ( \eta f + \coms + \S \comcomp)}{\rho} \, \frac{d\rho}{\rho}\\
& + \Cb \int_r^{\infty} \exp \biggl( 
\Ca \int_r^{\rho} \lcw(\nu) \, \frac{d\nu}{\nu}
\biggr) \, \Np{\nabla (\eta f + \coms + \comcomp)}{\rho} \, \frac{d\rho}{\rho}.
\end{aligned}
\end{equation}
Since also~$\S u = f$ for~$|x| \leq 2r_0$, equation~(\ref{eq:Sw}) can be rewritten as
\[
\S w = \eta f + \S \eta u - \eta \S u + \S \comcomp = \eta f + \S \eta u - \eta f + \S \comcomp = \S ( \eta u + \comcomp) .
\]
In other words,~$\S (w - \eta u - \comcomp) = 0$, and since~$w - \eta u - \comcomp$ satisfies the conditions in Theorem~\ref{t:i:uniq},
it follows that~$w = \eta u + \comcomp$. Thus, 
\[
\Np{u}{r} = \Np{\eta u + \comcomp}{r} = \Np{w}{r}, \quad 0 < r < r_0.
\]

We now turn to prove~(\ref{eq:i:assymp}) in Theorem~\ref{t:i:local}.
%
We split the right-hand side of~(\ref{eq:est_w}) in two parts: one which deals with~$\coms + \S \comcomp$, and one for~$f$. 
From Lemma~\ref{l:est_Np_comm}, we obtain that
\[
\begin{aligned}
\int_0^r \left(  \frac{\rho}{r} \right)^{\Nh} \Np{\nabla (\coms + \S \comcomp)}{\rho} \, \frac{d\rho}{\rho} 
&\leq
C \| u \|_{\Xp} r^{-\Nh} \int_0^{r} \rho^{\Nh}(\rho + \lcw(\rho)) \, \frac{d\rho}{\rho}\\
& \leq C(r + \lcw(r)) \| u \|_{\Xp}.
\end{aligned}
\]
Let~$D(a,b)$ be defined by
\[
\intl(a,b) = \exp \biggl(  \Ca \int_a^{b} \lcw(\nu) \, \frac{d\nu}{\nu}  \biggr), \quad a,b \geq 0.
\]
It is true that
\[
\int_r^{r_0} \intl(r,\rho) \, \Np{\nabla (\coms + \S \comcomp)}{\rho} \, \frac{d\rho}{\rho} 
\leq
C \| u \|_{\Xp} \int_{r}^{r_0} (\rho + \lcw(\rho)) \intl(r,\rho) \, \frac{d\rho}{\rho}.
\]
Since 
\[
\int_r^{r_0} \rho \intl(r,\rho) \, \frac{d\rho}{\rho}
= 
\intl(r,r_0) \int_{r}^{r_0} \intl(r_0,\rho) \, d\rho
\leq C \intl(r,r_0)
\]
and
\[
\int_r^{r_0} \lcw(\rho) \intl(r,\rho) \, \frac{d\rho}{\rho}
\leq
\frac{1}{\Ca} \int_r^{r_0} \frac{\partial}{\partial \rho} \intl(r,\rho) \, d\rho
\leq
\frac{\intl(r,r_0)}{\Ca},
\]
we obtain that
\[
\int_r^{r_0} \intl(r,\rho) \, \Np{\nabla (\coms + \S \comcomp)}{\rho} \, \frac{d\rho}{\rho} 
\leq
C \| u \|_{\Xp} \intl(r,r_0).
\]
It is clear that~$r + \lcw(r)$ is small whenever~$r$ is small, and also that~$\intl(r,r_0) \geq 1$ for~$r < r_0$.
Hence,~$r + \lcw(r) \leq C\intl(r,r_0)$ if~$r < r_0$. 
Finally,
\[
\begin{aligned}
\int_{r_0}^{\infty} \intl(r,{\rho}) \,  \Np{\nabla (\coms + \S \comcomp)}{\rho} \, \frac{d\rho}{\rho} 
& \leq
C \| u \|_{\Xp} \int_{r_0}^{\infty} e^{-N\log \rho} \intl(r,\rho) \, \frac{d\rho}{\rho} \\
& \leq
\frac{C}{r_0^N(N - \Ca \lcwz)} \| u \|_{\Xp} \intl(r,r_0), 
\end{aligned}
\]
where we used the argument from~(\ref{eq:aa2}) in Lemma~\ref{l:est_kern_Ku} to estimate the last integral.

Now, since~$\nabla(\eta f) = \nabla \eta f + \eta \nabla f$, and~$\nabla \eta$ only lives for~$r_0 < |x| < 2r_0$,
it follows that the contribution from~$f$ is bounded by
\[
\begin{aligned}
& Cr^{-\Nh} \int_0^r {\rho}^{\Nh} \, \Np{\nabla f}{\rho} \, \frac{d\rho}{\rho}
+ C \intl(r,r_0) \int_r^{2r_0} \intl(r_0,\rho) \, \Np{\nabla f}{\rho} \, \frac{d\rho}{\rho}\\
& \qquad \qquad + C \intl(r,r_0) \int_{r_0/2}^{2r_0} \Np{f}{\rho} \, \frac{d\rho}{\rho}.\\
\end{aligned}
\]
This completes the proof.

\subsection{The Case of Summable~\boldmath{$\lcw$}}
Let the conditions of Theorem~\ref{t:i:local} be satisfied and
suppose that~$(\nu \mapsto \lcw(\nu)/\nu)$ belongs to~$L^1(0,2r_0)$. 
Then~$\intl(r,\rho) \leq C$ for~$0 \leq r \leq \rho \leq 1$, and from Theorem~\ref{t:i:local} it follows that
\[
\begin{aligned}
\Np{u}{r} &\leq C \int_0^r \left( \frac{\rho}{r} \right)^{\Nh} \, \Np{\nabla f}{\rho} \, \frac{d\rho}{\rho}
+ C \int_r^{2r_0} \Np{\nabla f}{\rho} \, \frac{d\rho}{\rho} + C(u,f)\\
&= 
 C \int_0^{2r_0} \Qw{\Nh}{0} \left( \frac{\rho}{r} \right) \, \Np{\nabla f}{\rho} \, \frac{d\rho}{\rho}
+ C(u,f).\\
\end{aligned}
\]
If also~$(\rho \mapsto \Np{\nabla f}{\rho}/\rho) \in L^1(0,2r_0)$,
then~$\Np{u}{r} \leq C$ for~$r < r_0$, where~$C$ depends on~$f$.

\subsection{Example:~\boldmath{$\Np{\nabla f}{r} \leq C r^{-\alpha}$}}
\label{s:est_Np_ralpha}
Let~$0 < \alpha < \Nh$.
Suppose that~$\lcw(r) \rightarrow 0$ as~$r \rightarrow 0$. Then there exists~$r_1 > 0$ such
that~$\lcw(r) \leq \alpha/(2\Ca)$ for~$0 < r < r_1$. 
Now, suppose that~$\Np{\nabla f}{r} \leq C r^{-\alpha}$ for~$r < 2r_0$. 
Let~$r < r_0$. Then Theorem~\ref{t:i:local} implies that
\begin{equation*}
\label{eq:assymp_ex1}
\begin{aligned}
\Np{u}{r} \leq {} & 
 \frac{C}{\Nh - \alpha}  r^{-\alpha}
+ C \int_r^{2r_0} \rho^{-\alpha} \intl(r,\rho) \, \frac{d\rho}{\rho}
+ C(u,f) \intl(r,r_0) ,
\end{aligned}
\end{equation*}
where 
\[
C(u,f) = C \biggl( \|u \|_{\Xp} + \int_{r_0/2}^{2r_0} \Np{f}{\rho} \, \frac{d\rho}{\rho} \biggr)
\]
and~$C$ only depends on~$N$ and~$p$.
Now, 
\[
\int_r^{2r_0} \rho^{-\alpha} \intl(r,\rho) \, \frac{d\rho}{\rho} =
r^{-\alpha} \int_r^{2r_0} \left( \frac{r}{\rho} \right)^{\alpha} \intl(r,\rho) \, \frac{d\rho}{\rho}. 
\]
Suppose that~$r < r_1$. 
We split the domain of integration in two parts. For the first part, we obtain that
\[
\begin{aligned}
\int_r^{r_1} \left( \frac{r}{\rho} \right)^{\alpha} \intl(r,\rho) \, \frac{d\rho}{\rho} &=
\int_r^{r_1} \exp\left( \Ca \int_r^{\rho} \left( \lcw(\nu) - \frac{\alpha}{\Ca} \right) \frac{d\nu}{\nu} \right) \, \frac{d\rho}{\rho}\\
&\leq \int_{r}^{r_1} \left( \frac{r}{\rho} \right)^{\alpha/2} \frac{d\rho}{\rho} \leq \frac{2}{\alpha}
\end{aligned} 
\]
since~$\lcw(\nu) - \alpha/\Ca \leq - \alpha/(2\Ca)$.
For the second part,
\[
\begin{aligned}
\int_{r_1}^{2r_0}  \left( \frac{r}{\rho} \right)^{\alpha} \intl(r,\rho) \, \frac{d\rho}{\rho} &=
\int_{r_1}^{2r_0}  \left( \frac{r}{r_1} \right)^{\alpha}  \intl(r,r_1) \left( \frac{r_1}{\rho} \right)^{\alpha} \intl(r_1,\rho) \, \frac{d\rho}{\rho}\\
&\leq \int_{r_1}^{2r_0} \left( \frac{r}{r_1} \right)^{\alpha/2} \left( \frac{r_1}{\rho} \right)^{\alpha} \intl(r_1,\rho) \, \frac{d\rho}{\rho}\\
&\leq \frac{r_1^{-\alpha/2} \intl(r_1,2r_0)}{\alpha} r^{\alpha/2},
\end{aligned} 
\]
similarly with the estimate for the first part.
Hence,
\[
\int_r^{2r_0} \rho^{-\alpha} \intl(r,\rho) \, \frac{d\rho}{\rho} \leq \frac{2}{\alpha} r^{-\alpha} + \frac{r_1^{-\alpha/2} \intl(r_1,2r_0)}{\alpha} r^{-\alpha/2} \leq C_{\alpha} r^{-\alpha}.
\]
Thus,
\[
\begin{aligned}
\Np{u}{r} \leq {} & 
C({\alpha}) r^{-\alpha} \leq C(u,f,{\alpha}) r^{-\alpha}
+ C(u,f) \intl(r,r_0) ,
\end{aligned}
\]
where~$C({\alpha})$ is a constant that depends on~$\alpha$. The last inequality follows from the fact that
it is possible to estimate~$\intl(r,r_0) \leq C({\epsilon})r^{-\epsilon}$ for every~$\epsilon > 0$, which follows 
analogously with the argument used above.


\def\bibname{References}

\end{document}